\documentclass[reqno]{amsart}
\usepackage{graphicx}
\graphicspath{ {./figures/} }
\usepackage[abs]{overpic}
\usepackage{amssymb, mathtools}
\usepackage{url}
\usepackage[colorlinks,allcolors={blue}]{hyperref}
\usepackage{caption}
\usepackage{subcaption}
\usepackage{color}
\usepackage[table, dvipsnames]{xcolor}
\usepackage{rotating}
\usepackage[OT2,OT1]{fontenc} 

\newcommand\cyr{%
\renewcommand\rmdefault{wncyr}%
\renewcommand\sfdefault{wncyss}%
\renewcommand\encodingdefault{OT2}%
\normalfont
\selectfont}
\DeclareTextFontCommand{\textcyr}{\cyr}

	\newcommand{\Z}{\mathbb{Z}}			
	\newcommand{\R}{\mathbb{R}}			
	\newcommand{\lob}{{\textcyr{L}}}		

	\theoremstyle{plain}
	\newtheorem{thm}{Theorem}[section]
	
	\newtheorem*{CRtheorem}{Casson-Rivin Theorem}
    \newtheorem{lem}[thm]{Lemma}
	
	\newtheorem{cor}[thm]{Corollary}
	
	\newtheorem{prop}[thm]{Proposition}

	\theoremstyle{definition}
	\newtheorem{defn}{Definition}[section]
	\newtheorem{observation}{Observation}
	\newtheorem{property}{Property}

\begin{document}

\title{Geometric triangulations of a family of hyperbolic 3-braids}
\author{Barbara Nimershiem}
\address{Department of Mathematics, Franklin \& Marshall College, Lancaster, PA  17604}
\email{\href{mailto:barbara.nimershiem@fandm.edu}{barbara.nimershiem@fandm.edu}}
\date{\today}                

\begin{abstract}
We construct topological triangulations for complements of $(-2,3,n)$-pretzel knots and links with $n\ge7$.  Following a procedure outlined by Futer and Gu\'eritaud, we use a theorem of Casson and Rivin to prove the constructed triangulations are geometric.  Futer, Kalfagianni, and Purcell have shown (indirectly) that such braids are hyperbolic.  The new result here is a direct proof.  
\end{abstract}

\maketitle

\section{Introduction}\label{sec:intro}

A knot or link in $S^3$ is \emph{hyperbolic} if its complement admits a complete hyperbolic structure.  A braid on $n$ strands, which is represented by a word in the braid group, $w\in B_n$, is hyperbolic if its braid closure, $L_w$, is.  Hyperbolic braids formed from three strands have been characterized by Futer, Kalfagianni, and Purcell \cite{FareyManifolds}, a characterization that depends on their representation in the braid group.  $B_3$ has two generators, $\sigma_1$ and $\sigma_2$ shown in Figure~\ref{fig:B3gen}, and one relation $\sigma_1\sigma_2\sigma_1 = \sigma_2\sigma_1\sigma_2 $.  The square of this element is commonly denoted by $C$, which is central.  In \cite[Theorem~5.5]{FareyManifolds}, Futer, Kalfagianni, and Purcell have shown that $S^3- L_w$ is hyperbolic if and only if $w$ is conjugate to $C^k\sigma_1^{p_1}\sigma_2^{-q_1}\cdots\sigma_1^{p_s}\sigma_2^{-q_s}$  with $k\in\Z$ and $p_i$, $q_i$, and $s$ positive integers, and $w$ is not conjugate to $\sigma_1^{p_0}\sigma_2^{q_0}$ for some integers~$p_0$ and~$q_0$ (closures of such braids are either unknots, torus knots, or connected sums of torus knots).  The proof is by contradiction; they show the required hyperbolic structures exist without  constructing associated triangulations.  In this paper, we construct geometric triangulations for braids with $k=2$, $s=1$, $p_1\ge1$, and $q_1=1$ thus offering a direct proof of the \lq\lq if'' direction of their theorem in these cases.  
\begin{figure}[ht]\vskip.25in
\begin{overpic}[unit=.25in]{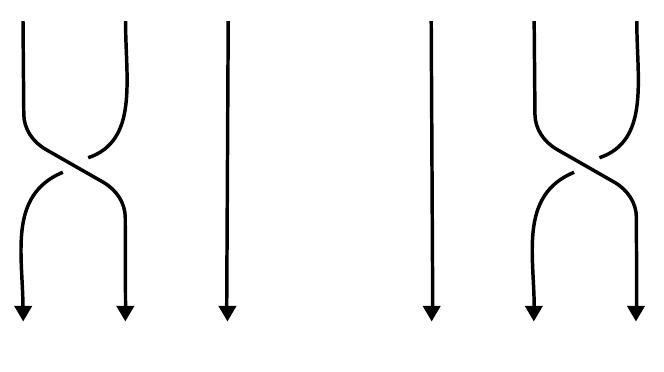}
\put(1.8,.3){\small{$\sigma_1$}}
\put(8.35,.3){\small{$\sigma_2$}}
\end{overpic}
\caption{The generators of $B_3$, the braid group on three strands.\label{fig:B3gen}}
\end{figure}

Let $L_p$ be the braid closure of $C^2\sigma_1^p\sigma_2^{-1}$ with $p\ge1$, which has one component if $p$ is odd and two if $p$ is even.  The complement of $L_p$ is homeomorphic to the complement of a $(-2,3,p+6)$-pretzel (Corollary~\ref{cor:pretzelstoo}), so proving $C^2\sigma_1^p\sigma_2^{-1}$ is hyperbolic will cover those pretzel knots and links as well.  To show $C^2\sigma_1^p\sigma_2^{-1}$ is hyperbolic, we will first decompose its complement into ideal tetrahedra whose faces are identified in pairs  (Section~\ref{sec:idealtriangulationsq=1}).  These topological triangulations (one for each $p$) are very much in the spirit of the canonical decompositions first described for 2-bridge knot and link complements by Sakuma and Weeks in~\cite{SakumaWeeks}. It turns out that, as in the 2-bridge case detailed by Futer in~\cite[Appendix]{GwithF2Bridge}, our triangulations are not just topological; they are also \emph{geometric} (though not canonical).  In other words, each ideal tetrahedron can be given a hyperbolic shape of positive volume and, when the faces are paired by hyperbolic isometries, the resulting hyperbolic structure on $S^3-L_p$ is metrically complete.

To prove that our topological triangulations are geometric, we will employ the \lq\lq Casson-Rivin program'' laid out by Futer and Gu\'eritaud in~\cite{AngledtoHyperbolic} --- their survey of Casson and Rivin's technique for finding the hyperbolic structure on a 3-manifold whose boundary consists of tori.  Gu\'eritaud has implemented this program for punctured torus bundles and 4-punctured sphere bundles~\cite{GwithF2Bridge} and Futer for 2-bridge knots~\cite[Appendix]{GwithF2Bridge}, whereas Gu\'eritaud and Schleimer have applied it to layered solid tori and Dehn fillings~\cite{DehnFillings} and Ham and Purcell used the same procedure to determine geometric triangulations on highly twisted links~\cite{HighlyTwisted}.

Geometric triangulations can lead to better understandings of the geometry of finite-volume hyperbolic 3-manifolds, including simpler proofs.  However, it is not known whether every such 3-manifold admits a geometric triangulation.  Currently the list of infinite families of finite-volume hyperbolic 3-manifolds that have been shown to admit geometric triangulations is short: punctured torus bundles and 4-punctured sphere bundles~\cite{GwithF2Bridge}, 2-bridge knots and links~\cite[Appendix]{GwithF2Bridge}, and certain Dehn fillings of fully-augmented 2-bridge knots and links~\cite{HighlyTwisted}. The main result of this paper is that the $(-2,3,n)$-pretzel knots and links with $n\ge 7$ can be added to the list. 

\subsection*{Organization} 
After constructing the topological triangulations in Section~\ref{sec:idealtriangulationsq=1}, we offer definitions and a basic outline of the Casson-Rivin program in Section~\ref{sec:anglestructures}. 
The first step of the program is carried out in Section~\ref{sec:anglestructuresnonempty}, where we show that the space of positive angle structures for our triangulations (Definition~\ref{def:anglestructure}) is nonempty. The final step --- arguing that the critical point of the volume functional (Definition~\ref{def:volume}) is a positive angle structure --- is presented in Section~\ref{sec:proofq=1} with some comments about extending the construction in Section~\ref{sec:extendingconstruction}.

\subsection*{Acknowledgments} 
The explicit triangulations appearing here, as well as several more in the infinite family, were confirmed using the software packages Regina~\cite{Regina} and SnapPy \cite{SnapPy} and checked against the census of veering triangulations created by Giannopolous, Schleimer, and Segerman \cite{VeeringCensus}.  See the Regina data file at \href{https://tinyurl.com/ReginaFileC2sigma1psigma2inv}{https://tinyurl.com/ReginaFileC2sigma1psigma2inv} for the triangulations constructed in Section~\ref{sec:idealtriangulationsq=1}.

To make the graphics more accessible to those with color vision deficiencies, I have used two of Paul Tol's qualitative color schemes --- \emph{bright} (Figures~\ref{fig:fulltwist} and~\ref{fig:clasp}) and \emph{light} (Figures~\ref{fig:p=1} and~\ref{fig:firstlayer}).  His palettes are mathematically designed to appear distinct to all \cite{PaulTol}.

David Futer kindly gave feedback on an early draft of Appendix~\ref{app:seeing2-bridge} for which I am grateful.  
I especially want to thank Bill Dunbar for many helpful conversations.  There is no aspect of this paper --- the structure, the notation, the figures, the tables, the exposition, the results --- that has not been improved by his insights, careful reading, and self-described picky-ness.  I am also grateful for several useful suggestions offered by the referee.  That said, any remaining errors are mine.

\section{Ideal triangulations}  \label{sec:idealtriangulationsq=1}

Let $X_{p}=S^3-L_p$, the complement of the closure of the braid $C^2\sigma_1^p\sigma_2^{-1}$.  In this section, we construct a triangulation of $X_{p}$.  In other words, we decompose the space into tetrahedra with face pairings. The decomposition will be developed using ideas from Appendix~\ref{app:seeing2-bridge}, which contains a slight re-visualization of the geometric triangulations of 2-bridge link complements presented by Futer in~\cite[Appendix]{GwithF2Bridge}.  There he uses the fact that a 2-bridge link can be described as a 4-braid \lq\lq closed up'' with a clasp at each end.  The 4-braid lives in a product region, $S^2\times I$, and its complement in this region is also a product, $S\times I$, where $S$ is a 4-punctured sphere. In~\cite[Appendix]{GwithF2Bridge}, the product region, $S\times I$, appears between two nested pillowcases and Futer showed that the region can be triangulated using a sequence of layers of ideal tetrahedra.  In Appendix~\ref{app:seeing2-bridge}, we position the product region vertically.  As a result, the faces of Futer's layered tetrahedra can be easily seen in the braid complement, and the bottom of one layer is identified to the top of the next according to the half-twist between them.  (See, for example, Figures~\ref{fig:2-bridgelink} and~\ref{fig:firstlayer}.)  The full link complement can  be obtained by identifying the top of the product region to itself in a way that forms the  clasp needed at the top and similarly for the bottom (Figure~\ref{fig:clasp}).  These identifications amount to capping off each end of the $S\times I$ with a 3-ball from which the clasp has been removed.

Analogously, we observe that the closure of $C^2\sigma_1^p\sigma_2^{-1}$ can  be formed from a 6-braid \lq\lq closed up'' with a three-stranded full twist on both the top and the bottom.  In this context, the 6-braid $\sigma_1^p\sigma_2^{-1}$ can be replaced by $\sigma_1^p\sigma_4^{-1}$ because $\sigma_2^{-1}$ can slide past the central full twist $C$ on the bottom to become $\sigma_4^{-1}$ as indicated in Figure~\ref{fig:6braid}.  In this section, we will triangulate the complement, $X_{p}$, by placing ideal tetrahedra in the product region containing the 6-braid $\sigma_1^p\sigma_4^{-1}$ and identifying the topmost layer to itself in a way that forms a full twist on three strands and similarly at the bottom.  We begin with the challenge of how to identify a 6-punctured sphere to itself to form a full twist.

\begin{figure}[ht]\vskip.25in
\begin{overpic}[unit=.25in]{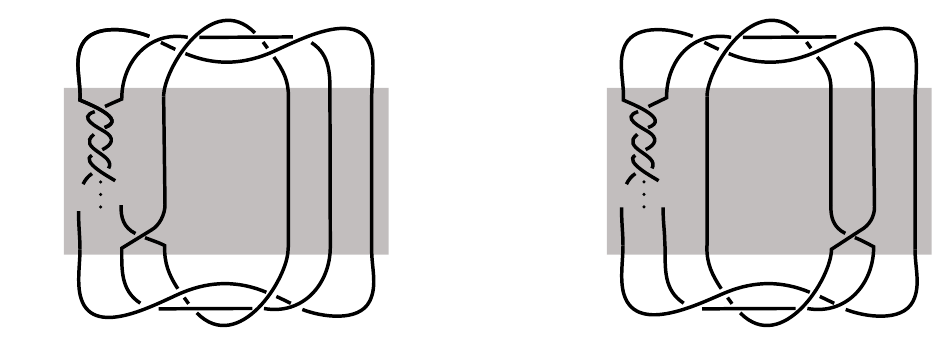}
\put(-.4,2.9){\small{$\sigma_1^p\sigma_2^{-1}$}}
\put(7.2,2.8){\small{$\approx$}}
\put(8.3,2.9){\small{$\sigma_1^p\sigma_4^{-1}$}}
\end{overpic}
\caption{The closure of $C^2\sigma_1^p\sigma_2^{-1}$ is the 6-braid $\sigma_1^p\sigma_2^{-1}$ --- or equivalently $\sigma_1^p\sigma_4^{-1}$ --- \lq\lq closed up'' with a three-stranded full twist on both the top and the bottom.\label{fig:6braid}}
\end{figure}

\subsection{Forming a three-stranded full twist}\label{sec:fulltwist}  

Let $S_0$ be the top 6-punctured sphere.  Label the punctures 0 through 5 and divide $S_0$ into eight triangles as shown in Figure~\ref{fig:fulltwist}(a).  Think of the punctures --- and the strands descending from them --- as living in the plane of the page, so $\triangle 025$ is in front of the page and $\triangle 035$ is behind the page.  Their shared edge, 05, passes through the point at infinity.  The colors indicate the identifications on $S_0$ that will yield a full twist:
\begin{description}
\item[gray] $\triangle 025$ is identified to $\triangle 035$ 
\item[white] $\triangle 023$ is identified to $\triangle 523$ 
\item[yellow]$\triangle 012$ in the front is identified to $\triangle 543$ in the front
\item[pink]$\triangle 012$ in the back is identified to $\triangle 543$ in the back
\end{description}
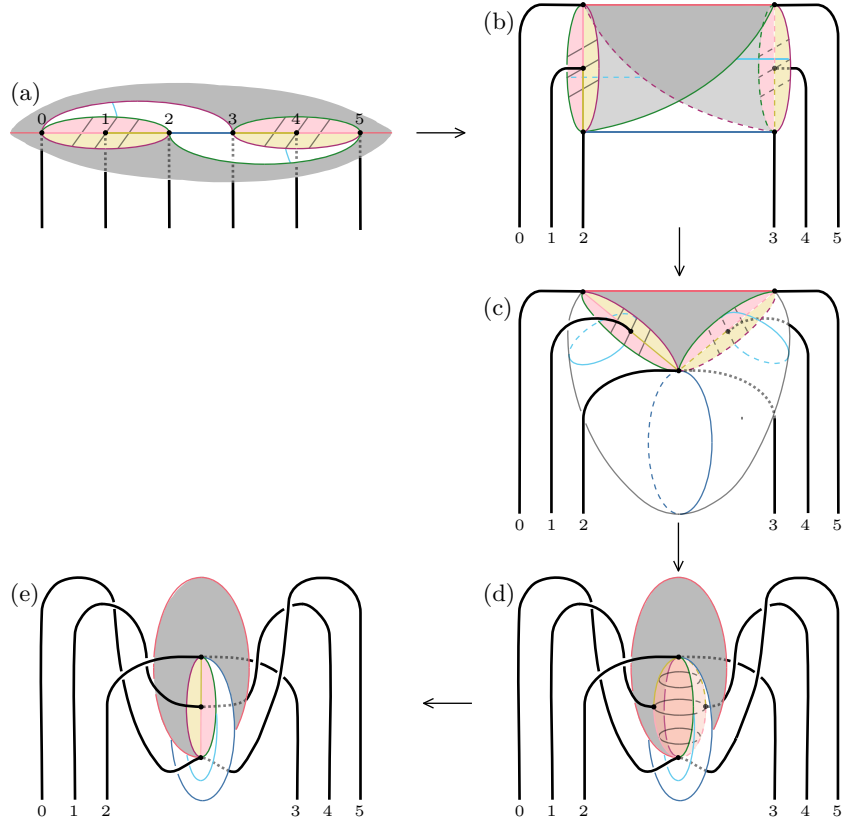
\begin{figure}[ht]\vskip.25in
\begin{overpic}[unit=.25in]{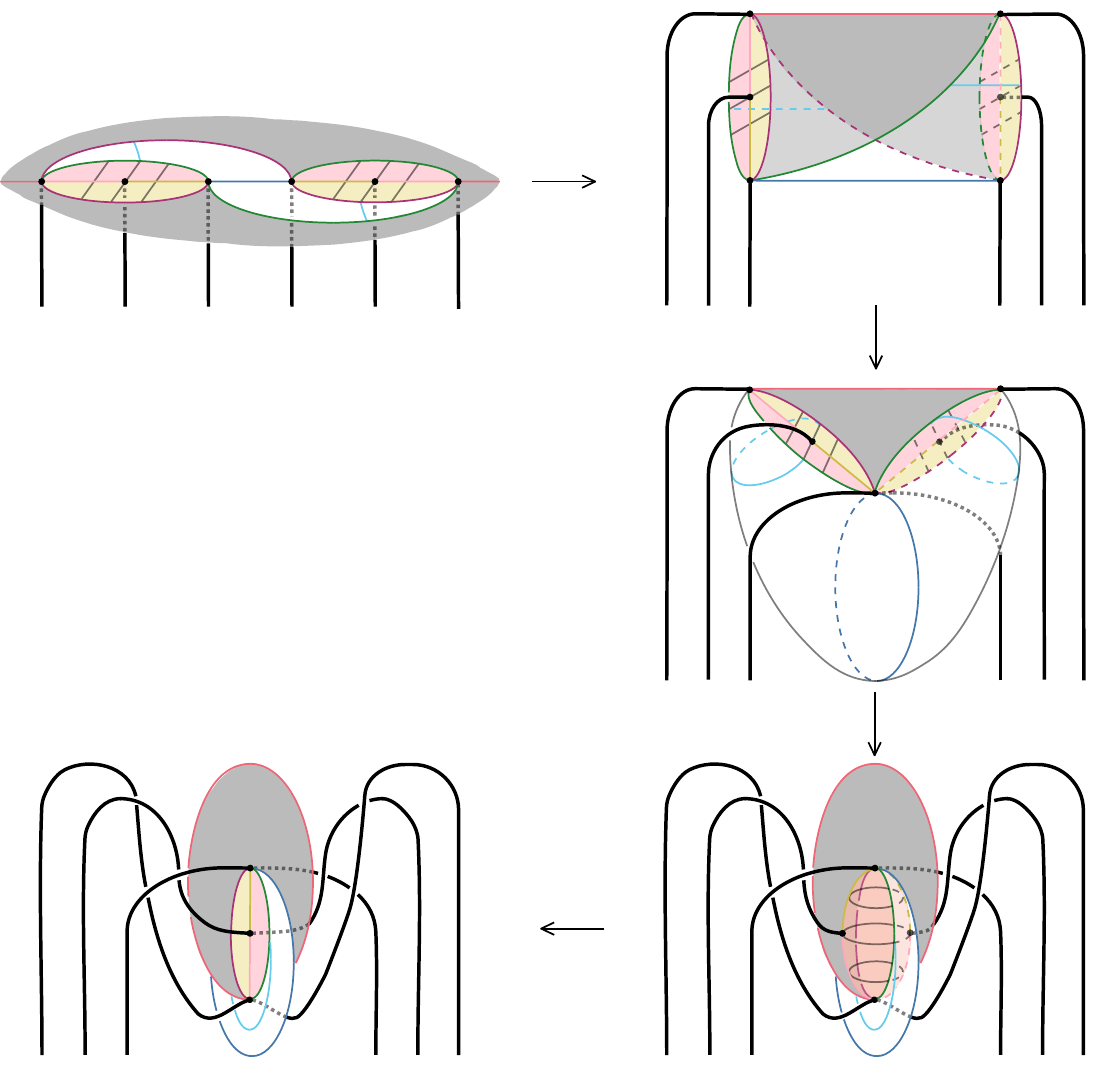}
\put(0.57, 14.5){\tiny0}\put(1.9, 14.5){\tiny1}\put(3.23, 14.5){\tiny2}\put(4.56, 14.5){\tiny3}\put(5.89, 14.5){\tiny4}\put(7.22, 14.5){\tiny5}

\put(10.57, 12){\tiny0}\put(11.24, 12){\tiny1}\put(11.91, 12){\tiny2}\put(15.88, 12){\tiny3}\put(16.55, 12){\tiny4}\put(17.22, 12){\tiny5}

\put(10.57, 6){\tiny0}\put(11.24, 6){\tiny1}\put(11.91, 6){\tiny2}\put(15.88, 6){\tiny3}\put(16.55, 6){\tiny4}\put(17.22, 6){\tiny5}

\put(10.57, 0){\tiny0}\put(11.24, 0){\tiny1}\put(11.91, 0){\tiny2}\put(15.88, 0){\tiny3}\put(16.55, 0){\tiny4}\put(17.22, 0){\tiny5}

\put(0.57, 0){\tiny0}\put(1.24, 0){\tiny1}\put(1.91, 0){\tiny2}\put(5.88, 0){\tiny3}\put(6.55, 0){\tiny4}\put(7.22, 0){\tiny5}

\put(0,15){\small(a)}
\put(9.9,16.5){\small(b)}
\put(9.9,10.5){\small(c)}
\put(9.9,4.5){\small(d)}
\put(0,4.5){\small(e)}

\end{overpic}
\caption{Forming a full twist $C$.\label{fig:fulltwist}}
\end{figure}
To see that identifying $S_0$ to itself in this way is the same as attaching a 3-ball with a full twist removed, first fold $S_0$ along the 23 edge into a cylindrical pillow shown in Figure~\ref{fig:fulltwist}(b).  As we go step-by-step through the identifications, we will follow the punctures (their paths are drawn as strands) and see they sweep out a full twist.  The extra markings --- cyan on the white triangles and gray on the yellow and pink ones --- will help us follow the triangles through each step.  
\begin{itemize}
\item $\triangle 025$ is identified to $\triangle 035$:  Figure~\ref{fig:fulltwist}(c) shows the result of bringing the 2 and 3 strands together to identify the gray triangles.  Placing this triangle in the plane of the page means the 23 (blue) edge forms a belt around a new pillow with half in front, as the 1 and 2 strands are, and half behind as the 3 and 4 strands are.  The left side of the pillow consists of $\triangle 023$ and both $\triangle 012$'s (appearing in the front).  The right side is formed by $\triangle 523$ and both $\triangle 543$'s (in the back).  This diagram is analogous to the third step in Figure~\ref{fig:clasp} in the appendix.
\item $\triangle 023$ is identified to $\triangle 523$:  By pushing the 0 and 5 strands into the pillow until they meet, $\triangle 023$ can be identified to $\triangle 523$ as shown in Figure~\ref{fig:fulltwist}(d).  The 0 strand goes behind 1 and 2 and the 5 strand passes in front of 3 and 4.  The result is not the same as the final step in Figure~\ref{fig:clasp}, because of the presence of the two $\triangle 012$'s and the two $\triangle 543$'s.  While these triangles themselves are not identified yet, their 02 edges and 53 edges (dark red and green) have come together, and they form a sphere with the two pink triangles in front and the two yellow in back.  The interlocked gray and white faces meet along the dark red and green longitudes.  These faces appear behind the sphere on the left and in front on the right.
\item The $\triangle 012$'s are identified to the $\triangle 543$'s:  Bringing strands 1 and 4 together and identifying each $\triangle 012$ to its corresponding $\triangle 543$ collapses the sphere.  The result, shown in Figure~\ref{fig:fulltwist}(e), is the desired positive full twist on three strands (read right to left, the positive direction for the top of the 3-braid with a counterclockwise orientation).  
\end{itemize}

To obtain the triangles for the 6-punctured sphere at the bottom of the product region, rotate $S_0$ about the horizontal line through the punctures, so now the strands of the braid go up, and use the same identifications.  For example, at the bottom $\triangle 023$ is in the front and is identified to $\triangle 523$ in the back.

Having capped off the ends, we turn our attention to placing ideal tetrahedra in the product region, determining their face pairings, and thus defining triangulations~$\tau_{p}$ on the manifolds $X_{p}$ for all~$p$.  We start by constructing a triangulation of $X_{1}$ upon which the other triangulations will be built. 

\subsection{A foundational ideal triangulation}   \label{sec:p=1}

To specify the triangulation of $X_{1}$, we will place tetrahedra in a product region from which the 6-braid, $\sigma_1\sigma_4^{-1}$, has been removed.   The left diagram in Figure~\ref{fig:p=1} shows this region with the braid in the plane of the page, except near the crossings.  

The method of Appendix~\ref{app:seeing2-bridge} suggests using three layers of tetrahedra (one on either side of the two crossings).  Label the top, middle, and bottom layers $\Delta_t$, $\Delta_m$, and $\Delta_b$ (with $\Delta_t$ split into $\Delta_{tt}$ and $\Delta_{tb}$).  The six ideal tetrahedra that will triangulate the complement are shown in Figure~\ref{fig:p=1}:  $t_1^\prime$~and~$t_2^\prime$ in $\Delta_t$, $m_1$ and $m_2$ in $\Delta_m$, and $b_1$ and $b_2^\prime$ in $\Delta_b$, where a prime indicates that the tetrahedron is behind the plane of the page; the labels on the ideal vertices come from the labels on their punctures (0 through 5 when read left to right); and the \lq\lq top'' faces of a tetrahedron are above the \lq\lq bottom'' faces relative to the product region.  

The edges of the tetrahedra are drawn both in the braid complement on the left and in flattened versions in the center and on the right.  On the left,  layers are expanded by making two copies of the edges where the top of a tetrahedron meets the bottom, allowing the faces of the tetrahedra and their remaining edges to be seen more easily.  The flattened versions show just one copy of each edge. 

The tetrahedra $t_1^\prime$ and $t_2^\prime$ share the face 025, shaded orange on the left, and their flattened versions are drawn separately.  The pair $m_1$ and $m_2$ share only an edge (thickened in the center diagram), so their interiors do not overlap when flattened, and they can both be seen when drawn side-by-side.  The same is true for~$b_1$ and~$b_2^\prime$.

The left figure contains additional segments that divide the boundaries of each layer into triangles.  These curves appear gray in the diagram and will be useful in determining the face pairings for this initial triangulation, which we shall refer to as~$\hat{\tau}_1$.
\begin{figure}[ht]\vskip.25in
\begin{overpic}[unit=.25in]{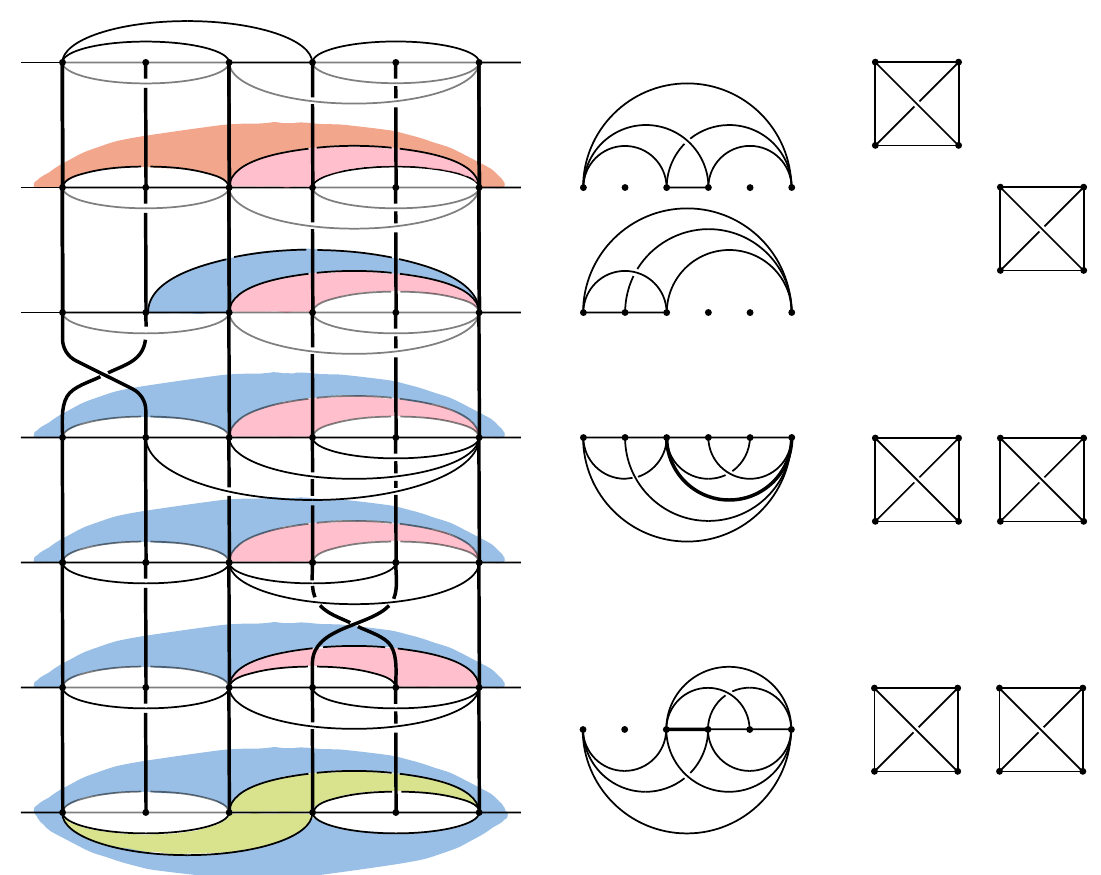}\put(0.7, 1.1){\tiny0}\put(2.03, 1.1){\tiny1}\put(3.36, 1.1){\tiny2}\put(4.69, 1.1){\tiny3}\put(6.02, 1.1){\tiny4}\put(7.35, 1.1){\tiny5}
\put(0.7, 3.1){\tiny0}\put(2.03, 3.1){\tiny1}\put(3.36, 3.1){\tiny2}\put(4.69, 3.1){\tiny3}\put(6.02, 3.1){\tiny4}\put(7.35, 3.1){\tiny5}
\put(0.7, 5.1){\tiny0}\put(2.03, 5.1){\tiny1}\put(3.36, 5.1){\tiny2}\put(4.69, 5.1){\tiny3}\put(6.02, 5.1){\tiny4}\put(7.35, 5.1){\tiny5}
\put(0.7, 7.1){\tiny0}\put(2.03, 7.1){\tiny1}\put(3.36, 7.1){\tiny2}\put(4.69, 7.1){\tiny3}\put(6.02, 7.1){\tiny4}\put(7.35, 7.1){\tiny5}
\put(0.7, 9.1){\tiny0}\put(2.03, 9.1){\tiny1}\put(3.36, 9.1){\tiny2}\put(4.69, 9.1){\tiny3}\put(6.02, 9.1){\tiny4}\put(7.35, 9.1){\tiny5}
\put(0.7, 11.1){\tiny0}\put(2.03, 11.1){\tiny1}\put(3.36, 11.1){\tiny2}\put(4.69, 11.1){\tiny3}\put(6.02, 11.1){\tiny4}\put(7.35, 11.1){\tiny5}
\put(0.7, 13.1){\tiny0}\put(2.03, 13.1){\tiny1}\put(3.36, 13.1){\tiny2}\put(4.69, 13.1){\tiny3}\put(6.02, 13.1){\tiny4}\put(7.35, 13.1){\tiny5}

\put(7.8,1.9){\small$\Delta_b=$}\put(7.8,5.9){\small$\Delta_m=$}\put(7.8,9.9){\small$\Delta_{tb}=$}\put(7.8,11.9){\small$\Delta_{tt}=$}

\put(13.1,1.9){\small$=$}\put(13.1,5.9){\small$=$}\put(13.1,9.9){\small$=$}\put(13.1,11.9){\small$=$}

\put(14.4,11.1){\small$t_1^\prime$}
\put(16.4,9.1){\small$t_2^\prime$}
\put(14.4,5.1){\small$m_1$}\put(16.4,5.1){\small$m_2$}
\put(14.4,1.1){\small$b_1$}\put(16.4,1.1){\small$b_2^\prime$}

\put(9.1, 11.1){\tiny0}\put(9.766, 11.1){\tiny1}\put(10.432, 11.1){\tiny2}\put(11.098,11.1){\tiny3}\put(11.764, 11.1){\tiny4}\put(12.43, 11.1){\tiny5}
\put(9.1, 9.1){\tiny0}\put(9.766, 9.1){\tiny1}\put(10.432, 9.1){\tiny2}\put(11.098,9.1){\tiny3}\put(11.764, 9.1){\tiny4}\put(12.43, 9.1){\tiny5}
\put(9.1, 7.1){\tiny0}\put(9.766, 7.1){\tiny1}\put(10.432, 7.1){\tiny2}\put(11.098, 7.1){\tiny3}\put(11.764, 7.1){\tiny4}\put(12.43, 7.1){\tiny5}
\put(9.1, 2.43){\tiny0}\put(9.766, 2.43){\tiny1}\put(10.432, 2.43){\tiny2}\put(11.098, 2.43){\tiny3}\put(11.764, 2.43){\tiny4}\put(12.43, 2.43){\tiny5}

\put(13.75, 13.1){\tiny0}\put(15.35, 13.1){\tiny5}
\put(13.75, 11.4){\tiny2}\put(15.35, 11.4){\tiny3}

\put(15.75, 11.1){\tiny0}\put(17.35, 11.1){\tiny5}
\put(15.75, 9.4){\tiny1}\put(17.35, 9.4){\tiny2}

\put(13.75, 7.1){\tiny1}\put(15.35, 7.1){\tiny2}\put(15.75, 7.1){\tiny3}\put(17.35, 7.1){\tiny4}
\put(13.75, 5.4){\tiny0}\put(15.35, 5.4){\tiny5}\put(15.75, 5.4){\tiny2}\put(17.35, 5.4){\tiny5}

\put(13.75, 3.1){\tiny2}\put(15.35, 3.1){\tiny3}\put(15.75, 3.1){\tiny2}\put(17.35, 3.1){\tiny5}
\put(13.75, 1.4){\tiny0}\put(15.35, 1.4){\tiny5}\put(15.75, 1.4){\tiny3}\put(17.35, 1.4){\tiny4}

\end{overpic}
\caption{A triangulation of the complement of the closure of $C^2\sigma_1\sigma_2^{-1}$.\label{fig:p=1}}
\end{figure}
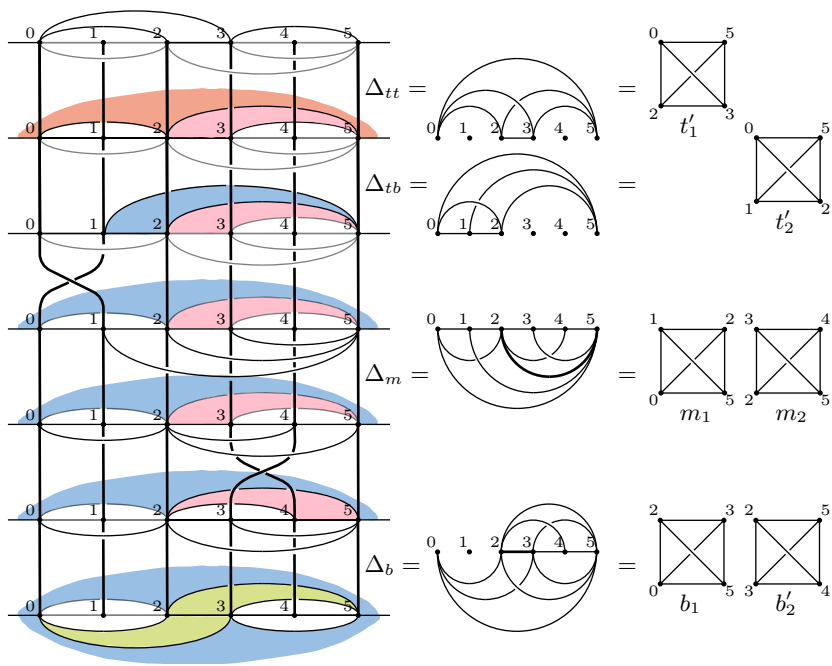

When comparing to Appendix~\ref{app:seeing2-bridge}, the reader should be concerned that the pairs of tetrahedra do not fill the entire layer.  In $\Delta_m$, for example, the tetrahedra are only in front of the plane of the page.  Nonetheless, if we slide the faces of these six tetrahedra around the link --- up and down the braid and through the full twists at the top and bottom --- they \emph{will} fill out the entire product region.  For example, face 235 on the bottom of $t_1^\prime$ can slide down the braid through $\Delta_{tb}$ and~$\Delta_m$ and past the $\sigma_4^{-1}$ half-twist to be identified with the 245 face of $b_2^\prime$.  We denote such faces and face pairings by:  
$$t_1^\prime(235)\sim b_2^\prime(245),$$
taking care that the order of the vertices indicates the correct match.  Images of the intersection of this triangle's path with each layer's boundary are shaded pink in the diagram.  They make use of the afore-mentioned additional gray segments and show that the back 235 region of $\Delta_m$ is indeed covered.  

As indicated in Figure~\ref{fig:p=1}, $t_1^\prime$ sits on top of $t_2^\prime$ with their  025 faces identified.  The remaining faces on the top of the tetrahedra in $\Delta_t$ and those on the bottom of the tetrahedra in $\Delta_b$ are triangles appearing in the previous section (\ref{sec:fulltwist}) and are identified by the full-twist identifications described there.  For example,  the face $b_1(023)$ on the bottom of $\Delta_b$ is the triangle $\triangle 023$, which is identified to $\triangle 523$, which is the 523 face of~$b_2^\prime$, another face on the bottom of~$\Delta_b$, so $b_1(023)\sim b_2^\prime(523)$.  These faces are shaded green.

Sometimes the full-twist identifications will combine with sliding along the braid.  To see such a combination, look at triangles shaded blue, starting with the 125 face of~$t_2^\prime$.  After traveling through the half-twist that interchanges 0 and 1, $t_2^\prime(125)$ becomes $\triangle 025$ on the top of $\Delta_m$ still in the back.  Then it can travel through $\Delta_m$ and $\Delta_b$ all the way to $\triangle 025$ on the bottom of $\Delta_b$ in the back.  This triangle is identified to $\triangle 035$ in the front, i.e., the 035 face of~$b_1$, so $t_2^\prime(125)\sim b_1(035)$.

The remaining face pairings of~$\hat{\tau}_1$ are obtained by sliding the remaining faces along the braid, twisting through half-twists and using full-twist identifications at the top and bottom as needed.  The reader may enjoy tracing the long and twisted journey of $t_2^\prime(012)$ to its mate, $m_1(021)$.  All twelve face pairings of~$\hat{\tau}_1$ are listed in Table~\ref{table:p=1prelim}.  

\begin{table}[ht]
\caption{Face pairings for~$\hat{\tau}_1$, an initial triangulation of $X_{1}$.\label{table:p=1prelim}\\}
\begin{tabular}{ |l c l|  }
 \hline
 $t_1^\prime(023)$    &$\sim$	&$m_2(523)$\\
  \hline
 $t_1^\prime(025)$    &$\sim$	&$t_2^\prime(025)$\\
  \hline
 $t_1^\prime(035)$    &$\sim$	&$m_1(125)$\\
  \hline
 $t_1^\prime(235)$    &$\sim$	&$b_2^\prime(245)$\\
  \hline
 $t_2^\prime(015)$    &$\sim$	&$m_1(105)$\\
  \hline
 $t_2^\prime(012)$    &$\sim$	&$m_1(021)$\\
  \hline
 $t_2^\prime(125)$    &$\sim$	&$b_1(035)$\\
  \hline
 $m_1(025)$    &$\sim$	&$b_1(025)$\\
  \hline
 $m_2(345)$    &$\sim$	&$b_2^\prime(354)$\\
  \hline
 $m_2(234)$    &$\sim$	&$b_2^\prime(243)$\\
  \hline
 $m_2(245)$    &$\sim$	&$b_1(235)$\\
  \hline
 $b_1(023)$    &$\sim$	&$b_2^\prime(523)$\\
 \hline
\end{tabular}
\end{table}

For reasons explained in~\ref{sec:veering}, the triangulation~$\hat{\tau}_1$ is not quite what we want for our triangulation of~$X_{1}$.  Instead, we  will simplify~$\hat{\tau}_1$ by applying Pachner moves (also known as bistellar flips, introduced in~\cite{Pachner}).  Table~\ref{table:p=1prelim} shows the following identifications of the 24 edge of $m_2$:
$$
24 \text{ in } m_2\sim 23 \text{ in } b_1\sim 23 \text{ in } b_2^\prime \sim_{\text{back to}}24 \text{ in } m_2.
$$
Whenever three tetrahedra surround an edge, we can perform a 3-2 Pachner move that replaces them with two tetrahedra sharing a face.  Figure~\ref{fig:3-2move} shows how to replace $m_2$, $b_1$, and $b_2^\prime$ with $\bar s$ on the top and $s$ on the bottom.  Label the vertices of $s$ with $abcd$ (circled in the diagram), where 
$a$~represents 3 in $m_2$ and 4 in $b_2^\prime$; 
$b$~represents 3 in $b_1$ and $b_2^\prime$ and 4 in $m_2$; 
$c$~represents 5 in $m_2$ and $b_1$; and 
$d$~represents 0 in $b_1$ and 5 in $b_2^\prime$.
The surviving faces of $m_2$, $b_1$, and $b_2^\prime$ are renamed as follows:
$m_2(235) \mapsto \bar s(2ac)$,
$m_2(345) \mapsto s(abc)$,
$b_1(025) \mapsto \bar s(d2c)$,
$b_1(035)  \mapsto s(dbc)$,
$b_2^\prime(245) \mapsto \bar s(2ad)$, and
$b_2^\prime(345) \mapsto s(bad)$.  
Making these substitutions and adding $\bar s(acd)\sim s(acd)$ yields the face pairings listed in Figure~\ref{fig:3-2move}.

\begin{figure}[ht]\vskip.25in
\begin{subfigure}{.4\textwidth}
\begin{overpic}[unit=.25in]{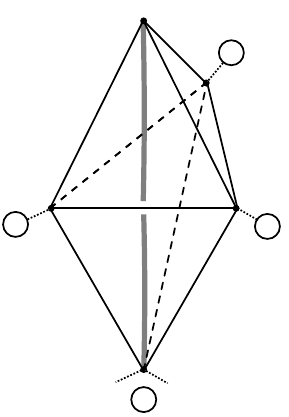}
\put(1,1.7){\small$s$}\put(.9,5){\small$\bar s$}
\put(1.8,3.2){\small$m_2$}
\put(3.2,4.2){\begin{turn}{60}\small$b_2^\prime$\end{turn}}
\put(1.6,4){\begin{turn}{40}\reflectbox{\small$b_1$}\end{turn}}

\put(4.2,3){\tiny$a$}\put(2.2,.15){\tiny$b$}\put(.15,3){\tiny$c$}\put(3.6,5.73){\tiny$d$}

\put(2.2,6.45){\tiny2}

\put(2,.6){\begin{turn}{-55}\tiny4\end{turn}}
\put(2.3,.4){\begin{turn}{60}\tiny4\end{turn}}
\put(1.89,.85){\begin{turn}{-55}\tiny3\end{turn}}
\put(2.43,.65){\begin{turn}{60}\tiny3\end{turn}}

\put(3.9,3.25){\begin{turn}{60}\tiny4\end{turn}}
\put(3.8,2.95){\begin{turn}{60}\tiny3\end{turn}}

\put(.4,3.4){\begin{turn}{-55}\tiny5\end{turn}}
\put(.53,3.1){\begin{turn}{-55}\tiny5\end{turn}}

\put(3.2,5.6){\begin{turn}{-215}\tiny0\end{turn}}
\put(3.4,5.45){\begin{turn}{-215}\tiny5\end{turn}}

\end{overpic}
\end{subfigure}
\begin{subfigure}{.4\textwidth}{\begin{tabular}{ |l c l|  }
 \hline
 $t_1^\prime(023)$    &$\sim$	&$\bar s(c2a)$\\
  \hline
 $t_1^\prime(025)$    &$\sim$	&$t_2^\prime(025)$\\
  \hline
 $t_1^\prime(035)$    &$\sim$	&$m_1(125)$\\
  \hline
 $t_1^\prime(235)$    &$\sim$	&$\bar s(2ad)$\\
  \hline
 $t_2^\prime(015)$    &$\sim$	&$m_1(105)$\\
  \hline
 $t_2^\prime(012)$    &$\sim$	&$m_1(021)$\\
  \hline
 $t_2^\prime(125)$    &$\sim$	&$s(dbc)$\\
  \hline
 $m_1(025)$    &$\sim$	&$\bar s(d2c)$\\
  \hline
 $s(abc)$    &$\sim$	&$s(bda)$\\
  \hline
 $\bar s(acd)$    &$\sim$	&$s(acd)$\\
  \hline
 \end{tabular}
}
\end{subfigure}\caption{A 3-2 move eliminates the 24 edge of $m_2$, which is also the 23 edge in $b_1$ and $b_2^\prime$, and yields the indicated face pairings.}\label{fig:3-2move} 
\end{figure}

Next, observe that the 23 edge of $t_1^\prime$ is identified to the $2a$ edge of $\bar s$, which is taken back to the 23 edge of $t_1^\prime$.  This very short cycle of edges means $t_1^\prime$ and $\bar s$ are identified across an edge, so we can perform a 2-0 Pachner move to collapse~$t_1^\prime$ and~$\bar s$.  This move identifies  $t_1^\prime(025)$ with $\bar s(c2d)$ and $t_1^\prime(035)$ with $\bar s(cad)$.  See Figure~\ref{fig:2-0move}, which shows how~$t_1^\prime$ and~$\bar s$ are stacked before they are collapsed and eliminated.  The 2-0 move results in a triangulation of $X_{1}$ comprising three tetrahedra --- $t_2^\prime,$ $m_1,$ and $s$ --- whose face pairings are in Figure~\ref{fig:2-0move}.  Define~$\tau_{1}$ to be this triangulation.   

\begin{figure}[ht]\vskip.25in
\begin{subfigure}{.4\textwidth}
\begin{overpic}[unit=.25in]{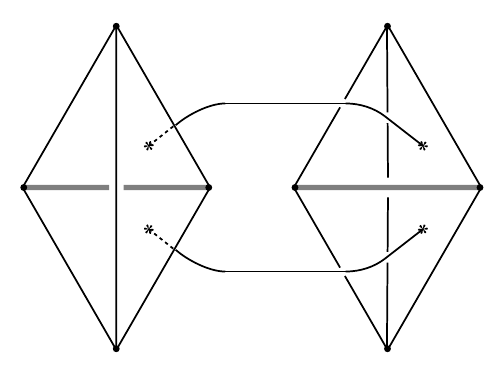}
\put(.6,1.2){\small$t_1^\prime$}\put(5,1.2){\small$\bar s$}

\put(1.75,5.7){\tiny5}\put(6.15,5.7){\tiny$d$}
\put(0,2.9){\tiny2}\put(3.4,2.9){\tiny3}\put(4.4,2.9){\tiny2}\put(7.8,2.9){\tiny$a$}
\put(1.75,0.1){\tiny0}\put(6.15,0.1){\tiny$c$}

\end{overpic}
\end{subfigure}
\quad
\begin{subfigure}{.4\textwidth}{\begin{tabular}{ |l c l|  }
 \hline
 $t_2^\prime(015)$    &$\sim$	&$m_1(105)$\\
  \hline
 $t_2^\prime(012)$    &$\sim$	&$m_1(021)$\\
  \hline
 $t_2^\prime(025)$    &$\sim$	&$m_1(520)$\\
 \hline 
 $t_2^\prime(125)$    &$\sim$	&$s(dbc)$\\
  \hline
 $m_1(125)$    &$\sim$	&$s(cad)$\\
 \hline
 $s(abc)$    &$\sim$	&$s(bda)$\\
  \hline
\end{tabular}
}
\end{subfigure}
\caption{A 2-0 move eliminates $t_1^\prime$ and $\bar s$, stacked as indicated, and yields face pairings for $\tau_1$, a triangulation of $X_1$.\label{fig:2-0move}}
\end{figure}

\subsection{Adding layers to extend to $p=2$ and $3$} \label{sec:p=3}

It will be more instructive to first extend the construction of the previous section to $p=3$.  To triangulate the complement of the closure of $C^2\sigma_1^3\sigma_2^{-1}$, where there are four half-twists in the braid, begin with five layers:  $\Delta_t$, $\Delta_m$, and $\Delta_b$ and two layers inserted between $\Delta_t$ and $\Delta_m$, labelled $\Delta_{w1}$ and $\Delta_{w2}$.  Place two tetrahedra in each layer as indicated in Figure~\ref{fig:p=3}.  Though $\Delta_t$ is depicted slightly differently (as one layer with the shared face of $t_1^\prime$ and $t_2^\prime$, 025, positioned in the middle), all tetrahedra in $\Delta_t$, $\Delta_m$, and~$\Delta_b$ are named and situated exactly as in~Figure~\ref{fig:p=1}.  The new tetrahedra, those in the~$\Delta_{wi}$ layers, will be called $w_i$ and $w_i^\prime$, where a prime still indicates the tetrahedron is behind the plane of the page.

\begin{figure}[ht]\vskip.25in
\begin{overpic}[unit=.25in]{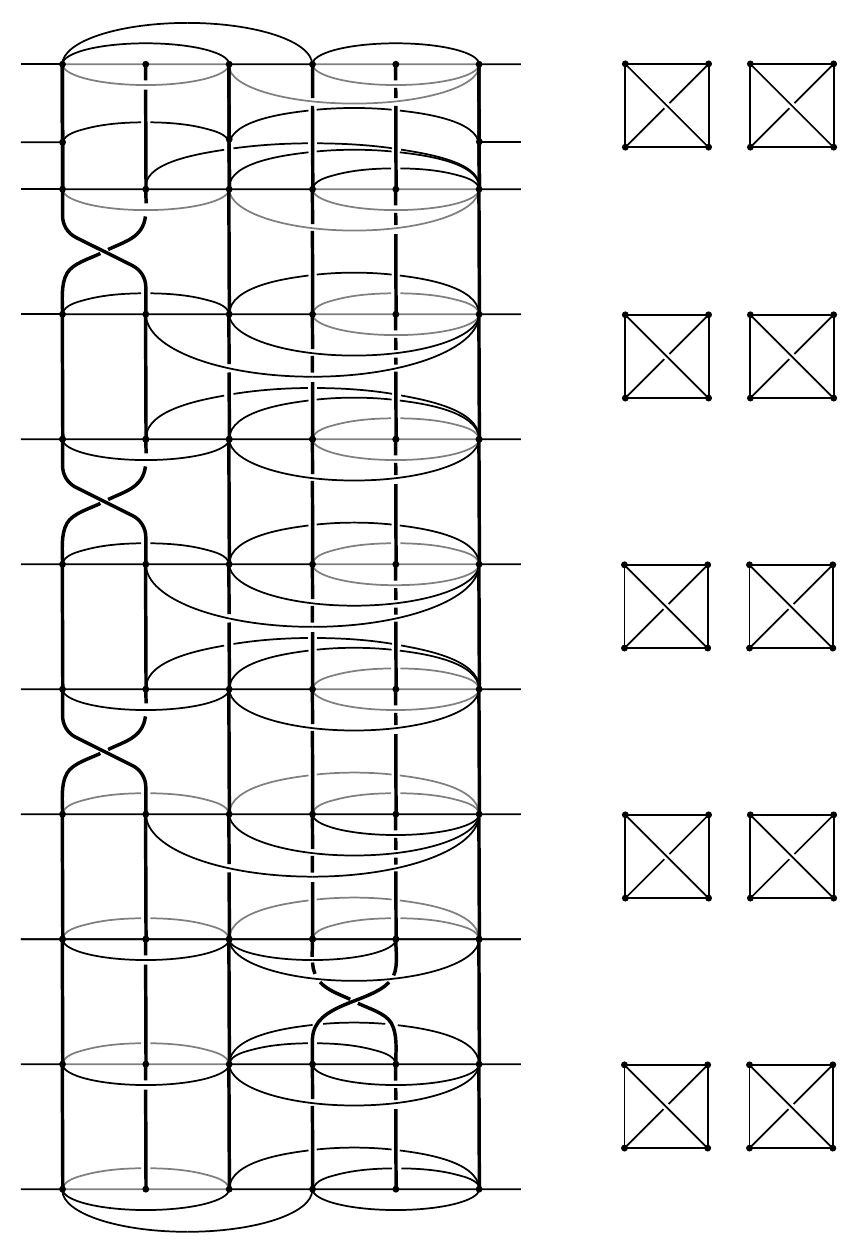}
\put(0.7, 1.1){\tiny0}\put(2.03, 1.1){\tiny1}\put(3.36, 1.1){\tiny2}\put(4.69, 1.1){\tiny3}\put(6.02, 1.1){\tiny4}\put(7.35, 1.1){\tiny5}
\put(0.7, 3.1){\tiny0}\put(2.03, 3.1){\tiny1}\put(3.36, 3.1){\tiny2}\put(4.69, 3.1){\tiny3}\put(6.02, 3.1){\tiny4}\put(7.35, 3.1){\tiny5}
\put(0.7, 5.1){\tiny0}\put(2.03, 5.1){\tiny1}\put(3.36, 5.1){\tiny2}\put(4.69, 5.1){\tiny3}\put(6.02, 5.1){\tiny4}\put(7.35, 5.1){\tiny5}
\put(0.7, 7.1){\tiny0}\put(2.03, 7.1){\tiny1}\put(3.36, 7.1){\tiny2}\put(4.69, 7.1){\tiny3}\put(6.02, 7.1){\tiny4}\put(7.35, 7.1){\tiny5}
\put(0.7, 9.1){\tiny0}\put(2.03, 9.1){\tiny1}\put(3.36, 9.1){\tiny2}\put(4.69, 9.1){\tiny3}\put(6.02, 9.1){\tiny4}\put(7.35, 9.1){\tiny5}
\put(0.7, 11.1){\tiny0}\put(2.03, 11.1){\tiny1}\put(3.36, 11.1){\tiny2}\put(4.69, 11.1){\tiny3}\put(6.02, 11.1){\tiny4}\put(7.35, 11.1){\tiny5}
\put(0.7, 13.1){\tiny0}\put(2.03, 13.1){\tiny1}\put(3.36, 13.1){\tiny2}\put(4.69, 13.1){\tiny3}\put(6.02, 13.1){\tiny4}\put(7.35, 13.1){\tiny5}
\put(0.7, 15.1){\tiny0}\put(2.03, 15.1){\tiny1}\put(3.36, 15.1){\tiny2}\put(4.69, 15.1){\tiny3}\put(6.02, 15.1){\tiny4}\put(7.35, 15.1){\tiny5}
\put(0.7, 17.1){\tiny0}\put(2.03, 17.1){\tiny1}\put(3.36, 17.1){\tiny2}\put(4.69, 17.1){\tiny3}\put(6.02, 17.1){\tiny4}\put(7.35, 17.1){\tiny5}
\put(0.7, 19.1){\tiny0}\put(2.03, 19.1){\tiny1}\put(3.36, 19.1){\tiny2}\put(4.69, 19.1){\tiny3}\put(6.02, 19.1){\tiny4}\put(7.35, 19.1){\tiny5}

\put(8,2){\small$\Delta_b=$}\put(8,6){\small$\Delta_m=$}\put(8,10){\small$\Delta_{w2}=$}\put(8,14){\small$\Delta_{w1}=$}\put(8,18){\small$\Delta_t=$}

\put(10.4,17.1){\small$t_1^\prime$}\put(12.4,17.1){\small$t_2^\prime$}
\put(10.4,13.1){\small$w_1$}\put(12.4,13.1){\small$w_1^\prime$}
\put(10.4,9.1){\small$w_2$}\put(12.4,9.1){\small$w_2^\prime$}
\put(10.4,5.1){\small$m_1$}\put(12.4,5.1){\small$m_2$}
\put(10.4,1.1){\small$b_1$}\put(12.4,1.1){\small$b_2^\prime$}

\put(9.75, 19.1){\tiny0}\put(11.35, 19.1){\tiny5}\put(11.75, 19.1){\tiny0}\put(13.35, 19.1){\tiny5}
\put(9.75, 17.4){\tiny2}\put(11.35, 17.4){\tiny3}\put(11.75, 17.4){\tiny1}\put(13.35, 17.4){\tiny2}

\put(9.75, 15.1){\tiny1}\put(11.35, 15.1){\tiny2}\put(11.75, 15.1){\tiny0}\put(13.35, 15.1){\tiny5}
\put(9.75, 13.4){\tiny0}\put(11.35, 13.4){\tiny5}\put(11.75, 13.4){\tiny1}\put(13.35, 13.4){\tiny2}

\put(9.75, 11.1){\tiny1}\put(11.35, 11.1){\tiny2}\put(11.75, 11.1){\tiny0}\put(13.35, 11.1){\tiny5}
\put(9.75, 9.4){\tiny0}\put(11.35, 9.4){\tiny5}\put(11.75, 9.4){\tiny1}\put(13.35, 9.4){\tiny2}

\put(9.75, 7.1){\tiny1}\put(11.35, 7.1){\tiny2}\put(11.75, 7.1){\tiny3}\put(13.35, 7.1){\tiny4}
\put(9.75, 5.4){\tiny0}\put(11.35, 5.4){\tiny5}\put(11.75, 5.4){\tiny2}\put(13.35, 5.4){\tiny5}

\put(9.75, 3.1){\tiny2}\put(11.35, 3.1){\tiny3}\put(11.75, 3.1){\tiny2}\put(13.35, 3.1){\tiny5}
\put(9.75, 1.4){\tiny0}\put(11.35, 1.4){\tiny5}\put(11.75, 1.4){\tiny3}\put(13.35, 1.4){\tiny4}

\end{overpic}
\caption{A triangulation of the complement of the closure of $C^2\sigma_1^3\sigma_2^{-1}$, with $\Delta_t$ depicted as one layer with the shared face of $t_1^\prime$ and $t_2^\prime$, 025, positioned in the middle.\label{fig:p=3}}
\end{figure}
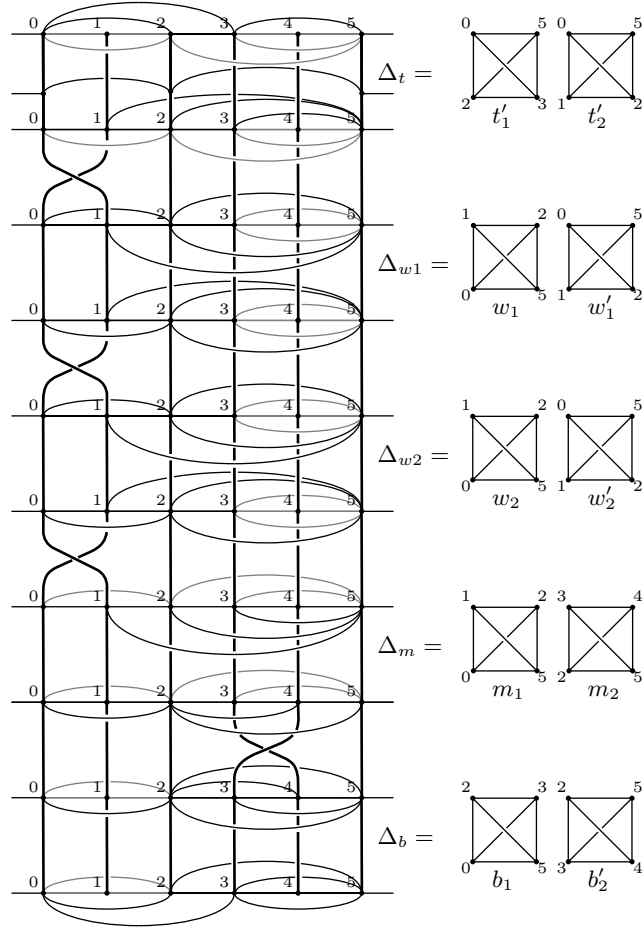

The face pairings are determined in the same manner as they were for $p=1$, but now over half of the twenty face pairings result from moving through a half-twist from the bottom of one layer to the top of the next.  These appear in the left column of Table~\ref{table:p=3prelim}.  The middle column contains the face pairings that come directly from Table~\ref{table:p=1prelim}  and are unaffected by adding the $w$-layers.  The final column contains more complicated face pairings that do not occur in Table~\ref{table:p=1prelim} because these triangles hit tetrahedra in the $w$-layers before encountering their mates  from Table~\ref{table:p=1prelim}.  For example,  $t_1^\prime(035)$ is identified to $\triangle 025$ in the front, which, after passing through the half-twist, becomes $\triangle 125$, a top face of $w_1$, so $t_1^\prime(035)$ hits  $w_1$ before it gets to $m_1$.  We will call this initial triangulation~$\hat{\tau}_3$, and, as in the $p=1$ case, will perform simplifying Pachner moves on $\hat{\tau}_3$ to create $\tau_{3}$.  The shadings in Table~\ref{table:p=3prelim} will help us determine the resulting identifications.
\begin{table}[ht]
\caption{Face pairings for $\hat{\tau}_3$, an initial triangulation of $X_3$.\label{table:p=3prelim}}
\definecolor{lightgray}{gray}{0.8}
\begin{tabular}{ |l c l|  }
  \hline
\rowcolor{White} $t_2^\prime(015)$    &$\sim$	&$w_1(105)$\\
  \hline
 \rowcolor{White} $t_2^\prime(125)$    &$\sim$	&$w_1^\prime(025)$\\
  \hline
\rowcolor{White}   $w_1(025)$    &$\sim$	&$w_2(125)$\\
  \hline
\rowcolor{White}  $w_1(012)$    &$\sim$	&$w_2^\prime(102)$\\
  \hline
\rowcolor{White} $w_1^\prime(015)$    &$\sim$	&$w_2(105)$\\
  \hline
\rowcolor{White}  $w_1^\prime(125)$    &$\sim$	&$w_2^\prime(025)$\\
  \hline
\rowcolor{White}  $w_2(025)$    &$\sim$	&$m_1(125)$\\
 \hline
\rowcolor{White}  $w_2^\prime(015)$    &$\sim$	&$m_1(105)$\\
 \hline
\rowcolor{lightgray}  $m_1(025)$    &$\sim$	&$b_1(025)$\\
  \hline
\rowcolor{Gray}  $m_2(234)$    &$\sim$	&$b_2^\prime(243)$\\
  \hline
\rowcolor{Gray}   $m_2(245)$    &$\sim$	&$b_1(235)$\\
  \hline
\end{tabular}\quad
\begin{tabular}{ |l c l|  }
 \hline
\rowcolor{Gray}   $t_1^\prime(023)$    &$\sim$	&$m_2(523)$\\
  \hline
 \rowcolor{lightgray}  $t_1^\prime(025)$    &$\sim$	&$t_2^\prime(025)$\\
  \hline
\rowcolor{Gray}    $t_1^\prime(235)$    &$\sim$	&$b_2^\prime(245)$\\
  \hline
 \rowcolor{White}  $t_2^\prime(012)$    &$\sim$	&$m_1(021)$\\
  \hline
\rowcolor{Gray}   $b_1(023)$    &$\sim$	&$b_2^\prime(523)$\\
 \hline
\end{tabular}\quad
\begin{tabular}{ |l c l|  }
 \hline
  \rowcolor{lightgray}  $t_1^\prime(035)$    &$\sim$	&$w_1(125)$\\
  \hline
 \rowcolor{lightgray}   $m_2(345)$    &$\sim$	&$w_1^\prime(201)$\\
  \hline
 \rowcolor{lightgray}  $b_1(035)$    &$\sim$	&$w_2^\prime(125)$\\
 \hline
\rowcolor{lightgray}   $b_2^\prime(345)$    &$\sim$	&$w_2(201)$\\
 \hline
\end{tabular}
\end{table}

Again, three tetrahedra in $\hat{\tau}_3$ surround the 24 edge of $m_2$ (which is identified to the 23 edges of $b_1$ and $b_2^\prime$), allowing for the same 3-2 move  shown in Figure~\ref{fig:3-2move}.  Thus  $m_2$, $b_1$, and $b_2^\prime$ are replaced by $\bar s$ and $s$.  As before, we can then perform a 2-0  move collapsing $t_1^\prime$ and $\bar s$ because they are identified across the edge $23=2a$.  Again,  $t_1^\prime(025)$ and $t_1^\prime(035)$ are identified with $\bar s(c2d)$ and $\bar s(cad)$ as in Figure~\ref{fig:2-0move}.  Let~$\tau_{3}$ be the resulting triangulation of $X_3$, consisting of the seven remaining tetrahedra: $t_2^\prime$, $w_1$, $w_1^\prime$, $w_2$, $w_2^\prime$, $m_1$, and~$s$.

The face pairings in the darkest cells in Table~\ref{table:p=3prelim}  involve only the four replaced or collapsed tetrahedra ($m_2$, $b_1$, $b_2^\prime$, $t_1^\prime$), so none of them will survive the Pachner moves defining~$\tau_{3}$.  The face pairings in the unshaded cells, on the other hand, do not involve any of these tetrahedra, and are, thus, unaffected by the two Pachner moves.  

The remaining face pairings of~$\tau_{3}$ come from those in the light gray cells in Table~\ref{table:p=3prelim}.  They involve the faces of $m_2$, $b_1$, $b_2^\prime$, and~$t_1^\prime$ that survive the Pachner moves.  In other words, these faces do appear in~$\tau_{3}$.  They are just relabelled by the Pachner moves.  Consider first $m_1(025)\sim b_1(025)$.  Under the 3-2 move shown in Figure~\ref{fig:3-2move},  $b_1(025)$ becomes  $\bar s(d2c)$, which is identified by the 2-0 move in Figure~\ref{fig:2-0move} to $t_1^\prime(520)$.  This face is  paired with $t_2^\prime(520)$, so, after the two Pachner moves, $m_1(025)$ is identified with $t_2^\prime(520)$, which we choose to write as $t_2^\prime(025)\sim m_1(520)$.  

The final four faces to be relabelled --- $t_1^\prime(035)$, $m_2(345)$, $b_1(035)$, and $b_2^\prime(345)$ --- are all faces that belong to the new tetrahedra $s$.  More specifically, the 2-0 move identifies $t_1^\prime(035)$ with $\bar s(cad)$, which is the new face introduced by the 3-2 move, the face shared by $\bar s$ and $s$.  Thus, after the two Pachner moves, $t_1^\prime(035)$ becomes $s(cad)$.  The other three faces form the rest of $s$ as indicated in Figure~\ref{fig:3-2move} and are thus the faces $s(abc)$, $s(dbc)$, and $s(bad)$, respectively.  Table~\ref{table:p=3} contains all identifications of~$\tau_{3}$, where the first column still lists face pairings that result from moving from the bottom of one layer to the top of the next and the final column now shows how the new tetrahedron, $s$, glues in. 

\begin{table}[ht]
\caption{Face pairings for~$\tau_{3}$, a triangulation of $X_3$.\label{table:p=3}}
\begin{tabular}{ |l c l|  }
   \hline
$t_2^\prime(015)$    &$\sim$	&$w_1(105)$\\
  \hline
$t_2^\prime(125)$    &$\sim$	&$w_1^\prime(025)$\\
  \hline
$w_1(025)$    &$\sim$	&$w_2(125)$\\
  \hline
 $w_1(012)$    &$\sim$	&$w_2^\prime(102)$\\
  \hline
$w_1^\prime(015)$    &$\sim$	&$w_2(105)$\\
  \hline
 $w_1^\prime(125)$    &$\sim$	&$w_2^\prime(025)$\\
  \hline
  $w_2(025)$    &$\sim$	&$m_1(125)$\\
 \hline
$w_2^\prime(015)$    &$\sim$	&$m_1(105)$\\
 \hline 
 \end{tabular}\quad
\begin{tabular}{ |l c l|  }
 \hline
 $t_2^\prime(012)$    &$\sim$	&$m_1(021)$\\
  \hline
 $t_2^\prime(025)$    &$\sim$	&$m_1(520)$\\
  \hline
 \end{tabular}
 \quad
\begin{tabular}{ |l c l|  }
 \hline
$s(cad)$    &$\sim$	&$w_1(125)$\\
  \hline
$s(abc)$    &$\sim$	&$w_1^\prime(201)$\\
  \hline
$s(dbc)$    &$\sim$	&$w_2^\prime(125)$\\
 \hline
$s(bad)$    &$\sim$	&$w_2(201)$\\
 \hline
\end{tabular}
\end{table}

The reader likely anticipates that we will create the desired triangulation of~$X_{p}$ in two steps:   First form an initial triangulation, $\hat{\tau}_p$, consisting of the tetrahedra in Figure~\ref{fig:p=1} together with some $w$-layers between $\Delta_t$ and $\Delta_m$ and then perform two Pachner moves --- a correct guess that will be made explicit in the next section.  Before moving on, we make some observations about the $p=3$ construction and address the $p=2$ case.

Figure~\ref{fig:p=3} shows that $t_2^\prime$ is situated and labelled in exactly the same way as $w_1^\prime$ and~$w_2^\prime$, so the labels on the bottom faces of $t_2^\prime$, 015 and 125, are the same as those on $w_1^\prime$ and~$w_2^\prime$.  Furthermore, as the face $t_2^\prime(015)$ slides through the half-twist, it rotates to the front, matching with $w_1(105)$, just as the 015 face of $w_1^\prime$ matches with $w_2(105)$.  Similarly, as the face $t_2^\prime(125)$ slides through the half-twist, it stays in the back, matching with $w_1^\prime(025)$, just as the the 125 face of $w_1^\prime$ matches with $w_2^\prime(025)$.  In these identifications, $t_2^\prime$ is acting as a tetrahedron labelled $w_0^\prime$ would.  A similar analysis at the bottom of the $w$-layers shows that the top faces of $m_1$ glue to the bottom faces of $w_2$ in the same way those in a tetrahedron labelled $w_3$ would.

Using this relabelling, we make several observations about the face pairings in~$\tau_{3}$.

\begin{itemize}

\item
By relabelling $t_2^\prime$ as $w_0^\prime$ and $m_1$ as $w_3$, the unshaded face pairings in the first column of Table~\ref{table:p=3prelim} can be summarized as follows:
\begin{center}
\begin{tabular}{ l c l l  }

$w_i^\prime(015)$    &$\sim$	&$w_{i+1}(105)$	&for $i=0,1,2$;\\

 $w_i^\prime(125)$    &$\sim$	&$w_{i+1}^\prime(025)$	&for $i=0,1$;\\

$w_i(025)$    &$\sim$	&$w_{i+1}(125)$	&for $i=1,2$; and\\

$w_i(012)$    &$\sim$	&$w_{i+1}^\prime(102)$	&for $i=1$.\\

 \end{tabular}
\end{center}
These inter-$w$ identifications of $\hat{\tau}_3$ are not affected by Pachner moves involving $m_2$, $b_1$, $b_2^\prime$, and~$t_1^\prime$, so they are also face pairings of~$\tau_{3}$, as shown in the first column of Table~\ref{table:p=3}.

\item
The top faces of $w_0^\prime$ are not included in the identifications listed above, and neither are the bottom faces of $w_3$.  These faces are identified to each other.  More specifically, after the Pachner moves, $w_0^\prime(012)\sim w_3(021)$ and $w_0^\prime(025)\sim w_3(520)$.  Both of these pairings also occur in~$\tau_{1}$ (see the table in Figure~\ref{fig:2-0move}) and were, thus, unaffected by the addition of the $w$-layers.  These face pairings appear in the second column of Table~\ref{table:p=3}.

\item
The face pairings in the third column of Table~\ref{table:p=3}, those involving~$s$, can be characterized as follows:  Two of the faces of $s$ are identified to top faces in the first $w$-layer, $w_1(125)$ and $w_1^\prime(201)$, and the other two are identified to bottom faces in the last $w$-layer, $w_2(201)$ and $w_2^\prime(125)$.
\end{itemize}

Properly interpreted, these observations also apply when $p=2$.  Let the triangulation~$\tau_{2}$ of $X_{2}$, the complement of the closure of $C^2\sigma_1^2\sigma_2^{-1}$, consist of $w_0^\prime$, $w_1$, $w_1^\prime$, $w_2$, and $s$.  Because there is only one $w$-layer, it is both the first $w$-layer and the last $w$-layer, and the final observation tells us that the faces of $s$ are identified to four faces in~$\Delta_{w1}$.  Harvesting the remaining face pairings from the first two observations, we obtain the face pairings for~$\tau_{2}$ given in Table~\ref{table:p=2}.

\begin{table}[ht]
\caption{Face pairings for~$\tau_{2}$, a triangulation of $X_{2}$.\label{table:p=2}}
\begin{tabular}{ |l c l|  }
   \hline
$w_0^\prime(015)$    &$\sim$	&$w_1(105)$\\
  \hline
$w_0^\prime(125)$    &$\sim$	&$w_1^\prime(025)$\\
  \hline
$w_1(025)$    &$\sim$	&$w_2(125)$\\
  \hline
$w_1^\prime(015)$    &$\sim$	&$w_2(105)$\\
  \hline
 \end{tabular}\quad
\begin{tabular}{ |l c l|  }
 \hline
 $w_0^\prime(012)$    &$\sim$	&$w_2(021)$\\
  \hline
 $w_0^\prime(025)$    &$\sim$	&$w_2(520)$\\
  \hline
 \end{tabular}\quad
\begin{tabular}{ |l c l|  }
 \hline
$s(cad)$    &$\sim$	&$w_1(125)$\\
  \hline
$s(abc)$    &$\sim$	&$w_1^\prime(201)$\\
  \hline
$s(dbc)$    &$\sim$	&$w_1^\prime(125)$\\
 \hline
$s(bad)$    &$\sim$	&$w_1(201)$\\
 \hline
\end{tabular}
\end{table}

\subsection{An ideal triangulation for any $p$} \label{sec:anyp}

Guided by the previous examples along with Figure~\ref{fig:p=3},  form an initial triangulation, $\hat{\tau}_p$,  of $X_{p}$, the complement of the closure of $C^2\sigma_1^p\sigma_2^{-1}$, that consists of $2p+4$ tetrahedra:  $t_1^\prime$ and $w_0^\prime$ (n\'ee~$t_2^\prime$) in $\Delta_t$; $w_p$ (n\'ee~$m_1$) and $m_2$ in $\Delta_m$; $b_1$ and $b_2^\prime$ in $\Delta_b$; and $w_i$ and $w_i^\prime$ in $\Delta_{wi}$ where $i=1,\dots, p-1$.  The face pairings of $\hat{\tau}_p$ are determined by sliding faces along the braid and using the full-twist identifications at the top and bottom of the product region as described in~\ref{sec:p=1} and~\ref{sec:p=3}.  As before, the addition of the $w$-layers does not affect identifications amongst $m_2$, $b_1$, $b_2^\prime$, and $t_1^\prime$, so the Pachner moves depicted in Figures~\ref{fig:3-2move} and~\ref{fig:2-0move} can be applied.  Define $\tau_{p}$ to be the resulting triangulation.  As with $\tau_{2}$, the observations about $\tau_{3}$ from the previous section apply to $\tau_{p}$, yielding the  face pairings listed in Table~\ref{table:anyp} for $p>1$.  (For $p=1$, see Figure~\ref{fig:2-0move}.)

\begin{table}[ht]
\caption{Face pairings for~$\tau_{p}$, a triangulation of $X_{p}$ for $p>1$.\label{table:anyp}}
\begin{tabular}{ |l c l l |}
   \hline
$w_i^\prime(015)$    &$\sim$	&$w_{i+1}(105)$ & $i=0,\dots,p-1$\\
  \hline
 $w_i^\prime(125)$    &$\sim$	&$w_{i+1}^\prime(025)$ & $i=0,\dots,p-2$\\
  \hline
$w_i(025)$    &$\sim$	&$w_{i+1}(125)$ & $i=1,\dots,p-1$\\
  \hline
 $w_i(012)$    &$\sim$	&$w_{i+1}^\prime(102)$ &$i=1,\dots,p-2$\\
  \hline
  $w_0^\prime(012)$    &$\sim$	&$w_p(021)$&\\
  \hline
$w_0^\prime(025)$    &$\sim$	&$w_p(520)$&\\
  \hline
$s(cad)$    &$\sim$	&$w_1(125)$ &\\
  \hline
$s(abc)$    &$\sim$	&$w_1^\prime(201)$&\\
  \hline
$s(dbc)$    &$\sim$	&$w_{p-1}^\prime(125)$&\\
 \hline
$s(bad)$    &$\sim$	&$w_{p-1}(201)$&\\
 \hline
 \end{tabular}
\end{table}

The triangulation $\hat{\tau}_p$ had $2p+4$ tetrahedra.  The 3-2 Pachner move reduced this number by one and the 2-0 Pachner move further reduced the number of tetrahedra by two, so $\tau_{p}$ has $2p+1$ tetrahedra and thus $4p+2$ face pairings.  There are exactly $4p+2$ identifications in Table~\ref{table:anyp} (and six in Figure~\ref{fig:2-0move}).  Thus, the definition of the triangulation~$\tau_{p}$ is complete.  

Note that most of the tetrahedra in~$\tau_{p}$ (all but $s$) are visible in the complement of the 6-braid $\sigma_1^p\sigma_4^{-1}$.  Using this visualization can help us quickly deduce relationships between the tetrahedra.  For example, Figure~\ref{fig:anyp} shows how $w_0^\prime$, $w_1$, $w_1^\prime$, $\dots$, $w_{p-1}$, $w_{p-1}^\prime$, and $w_p$ are situated in the braid complement, and all but six face pairings of~$\tau_{p}$ can be derived directly from following their faces through the half-twists in the braid.

\begin{figure}[ht]\vskip.25in
\begin{overpic}[unit=.25in]{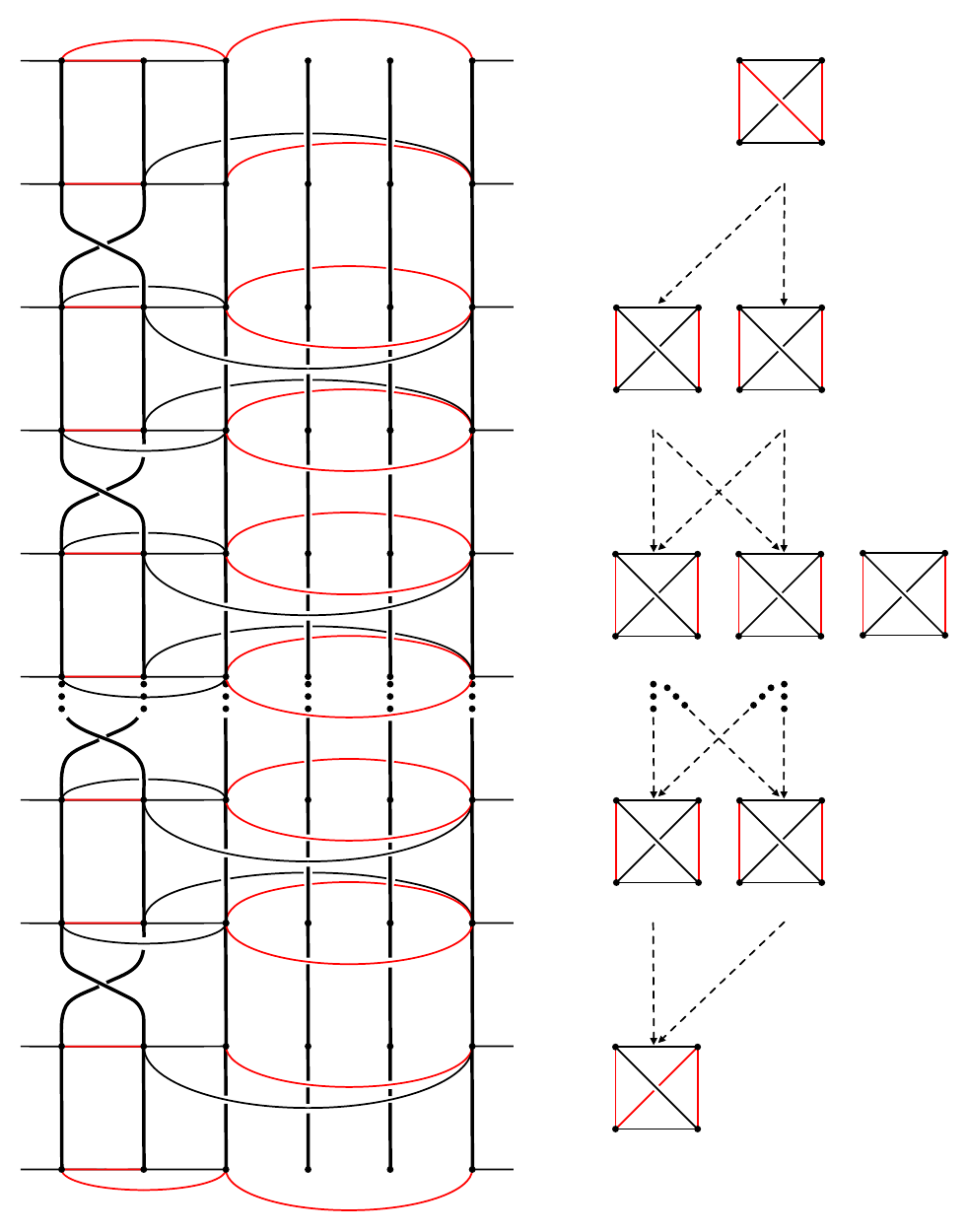}
\put(0.7, 1.1){\tiny0}\put(2.03, 1.1){\tiny1}\put(3.36, 1.1){\tiny2}\put(4.69, 1.1){\tiny3}\put(6.02, 1.1){\tiny4}\put(7.35, 1.1){\tiny5}
\put(0.7, 3.1){\tiny0}\put(2.03, 3.1){\tiny1}\put(3.36, 3.1){\tiny2}\put(4.69, 3.1){\tiny3}\put(6.02, 3.1){\tiny4}\put(7.35, 3.1){\tiny5}
\put(0.7, 5.1){\tiny0}\put(2.03, 5.1){\tiny1}\put(3.36, 5.1){\tiny2}\put(4.69, 5.1){\tiny3}\put(6.02, 5.1){\tiny4}\put(7.35, 5.1){\tiny5}
\put(0.7, 7.1){\tiny0}\put(2.03, 7.1){\tiny1}\put(3.36, 7.1){\tiny2}\put(4.69, 7.1){\tiny3}\put(6.02, 7.1){\tiny4}\put(7.35, 7.1){\tiny5}
\put(0.7, 9.1){\tiny0}\put(2.03, 9.1){\tiny1}\put(3.36, 9.1){\tiny2}\put(4.69, 9.1){\tiny3}\put(6.02, 9.1){\tiny4}\put(7.35, 9.1){\tiny5}
\put(0.7, 11.1){\tiny0}\put(2.03, 11.1){\tiny1}\put(3.36, 11.1){\tiny2}\put(4.69, 11.1){\tiny3}\put(6.02, 11.1){\tiny4}\put(7.35, 11.1){\tiny5}
\put(0.7, 13.1){\tiny0}\put(2.03, 13.1){\tiny1}\put(3.36, 13.1){\tiny2}\put(4.69, 13.1){\tiny3}\put(6.02, 13.1){\tiny4}\put(7.35, 13.1){\tiny5}
\put(0.7, 15.1){\tiny0}\put(2.03, 15.1){\tiny1}\put(3.36, 15.1){\tiny2}\put(4.69, 15.1){\tiny3}\put(6.02, 15.1){\tiny4}\put(7.35, 15.1){\tiny5}
\put(0.7, 17.1){\tiny0}\put(2.03, 17.1){\tiny1}\put(3.36, 17.1){\tiny2}\put(4.69, 17.1){\tiny3}\put(6.02, 17.1){\tiny4}\put(7.35, 17.1){\tiny5}
\put(0.7, 19.1){\tiny0}\put(2.03, 19.1){\tiny1}\put(3.36, 19.1){\tiny2}\put(4.69, 19.1){\tiny3}\put(6.02, 19.1){\tiny4}\put(7.35, 19.1){\tiny5}

\put(12.4,17.1){\small$w_0^\prime$}
\put(10.4,13.1){\small$w_1$}\put(12.4,13.1){\small$w_1^\prime$}
\put(10.4,9.1){\small$w_2$}\put(12.4,9.1){\small$w_2^\prime$}\put(14.5,9.1){\small$s$}
\put(10.4,5.1){\small$w_{p-1}$}\put(12.4,5.1){\small$w_{p-1}^\prime$}
\put(10.4,1.1){\small$w_p$}

\put(11.75, 19.1){\tiny0}\put(13.35, 19.1){\tiny5}
\put(11.75, 17.4){\tiny1}\put(13.35, 17.4){\tiny2}

\put(9.75, 15.1){\tiny1}\put(11.35, 15.1){\tiny2}\put(11.75, 15.1){\tiny0}\put(13.35, 15.1){\tiny5}
\put(9.75, 13.4){\tiny0}\put(11.35, 13.4){\tiny5}\put(11.75, 13.4){\tiny1}\put(13.35, 13.4){\tiny2}

\put(9.75, 11.1){\tiny1}\put(11.35, 11.1){\tiny2}\put(11.75, 11.1){\tiny0}\put(13.35, 11.1){\tiny5}\put(13.75, 11.1){\tiny d}\put(15.35, 11.1){\tiny c }
\put(9.75, 9.4){\tiny0}\put(11.35, 9.4){\tiny5}\put(11.75, 9.4){\tiny1}\put(13.35, 9.4){\tiny2}\put(13.75, 9.4){\tiny a}\put(15.35, 9.4){\tiny b}

\put(9.75, 7.1){\tiny1}\put(11.35, 7.1){\tiny2}\put(11.75, 7.1){\tiny0}\put(13.35, 7.1){\tiny5}
\put(9.75, 5.4){\tiny0}\put(11.35, 5.4){\tiny5}\put(11.75, 5.4){\tiny1}\put(13.35, 5.4){\tiny2}

\put(9.75, 3.1){\tiny1}\put(11.35, 3.1){\tiny2}
\put(9.75, 1.4){\tiny0}\put(11.35, 1.4){\tiny5}

\end{overpic}
\caption{The tetradedra of~$\tau_{p}$ with $[w_0^\prime(01)]$ colored red.  All but $s$ are shown in the braid complement on the left.  The dashed arrows on the right indicate that a bottom face of the originating tetrahedron is identified through a half-twist to a top face of the target tetrahedron.  Among faces unpaired by half-twists, the top faces of $w_0^\prime$ are identified to the bottom faces of $w_p$ and $s$ is glued to the top faces of $w_1$ and $w_1^\prime$ and the bottom faces of $w_{p-1}$ and $w_{p-1}^\prime$ as in Table~\ref{table:anyp}.\label{fig:anyp}}
\end{figure}

\subsection{Equivalence classes of edges in the triangulation} \label{sec:edges}

Knowing the relationships between the edges of $\tau_{p}$ will prove helpful because their equivalence classes form the edges of $X_{p}$ and the (internal) dihedral angles along these edges will dictate its geometry.  In determining equivalence classes, we can use Figure~\ref{fig:anyp} or Table~\ref{table:anyp}, whichever is easier.

As an example, we will determine all edges identified to the 01 edge of $w_0^\prime$.  In the braid picture, $w_0^\prime(01)$ slides down through the first $\sigma_1$ half-twist to be identified with $w_1(10)$ and $w_1^\prime(10)$, so the equivalence class of $w_0^\prime(01)$, $[w_0^\prime(01)]$, contains the 01 edge of $w_1$ and of $w_1^\prime$.  Continuing to move the~01 edge through more half-twists shows the 01 edges of $w_i$ and $w_i^\prime$ belong to $[w_0^\prime(01)]$ for all~$i$.  Two edges of $s$ also appear in this equivalence class, which can be seen most easily in Table~\ref{table:anyp}:  $w_1^\prime(01)\sim s(bc)$ and $w_{p-1}(01)\sim s(ad)$.  These edges of $s$, $s(bc)$ and $s(ad)$, are also identified to $w_{p-1}^\prime(25)$ and $w_1(25)$, respectively, which can themselves be slid along along the braid to show that the 25 edges of each $w_i$ and~$w_i^\prime$ are in~$[w_0^\prime(01)]$.  We have not yet considered face pairings between $w_0^\prime$ and $w_p$.  The first in Table~\ref{table:anyp}, $w_0^\prime(012)\sim w_p(021)$ shows $w_p(02)$ and $w_0^\prime(02)$ are also in~$[w_0^\prime(01)]$.  Based on how the tetrahedra are positioned in the braid complement, 02 is the top edge in $w_0^\prime$ and the bottom edge in $w_p$, which allows us to describe~$[w_0^\prime(01)]$ as follows.

\begin{observation}\label{obs:veeringred}
The equivalence class $[w_0^\prime(01)]$ contains a total of $4p+4$ edges:  the 01 edges and 25 edges of each $w_i$ and $w_i^\prime$; $s(bc)$ and $s(ad)$; the top edge of $w_0^\prime$ and the bottom edge of $w_p$.  These edges are colored red in Figure~\ref{fig:anyp}.
\end{observation}

\begin{observation}\label{obs:veeringblue5}
A similar analysis shows that $X_{p}$ has two edges with degree~5.  The edge $[w_0^\prime(05)]$ arises from 
$$
w_0^\prime(05)\sim w_1(15)\sim s(cd)\sim w_{p-1}^\prime(51)\sim w_p(05)\sim_{\text{back to}} w_0^\prime(05)
$$
and, when $p>1$, the edge $[w_0^\prime(12)]$ from
$$
w_0^\prime(12)\sim w_1^\prime(02)\sim s(ba)\sim w_{p-1}(20)\sim w_p(21)\sim_{\text{back to}} w_0^\prime(12).
$$
When $p=1$, $w_1^\prime(02)$ and $w_{p-1}(20)$ are undefined and are replaced by $s(db)$ and~$s(ac)$.
\end{observation}

It turns out that the remaining $2p-2$ edges have degree~4.

\section{Angle structures and the Casson-Rivin program}\label{sec:anglestructures}  

We will prove that the triangulation~$\tau_{p}$ of $X_{p}$ defined in~\ref{sec:anyp} is geometric by applying the Casson-Rivin program.  This section provides some definitions,  useful results, and a basic outline of the program.  For more details and a beautiful exposition, including an elementary proof of the Casson-Rivin Theorem, the reader is encouraged to read Futer and Gu\'eritaud's paper~\cite{AngledtoHyperbolic}.

\subsection{Angle structures} \label{sec:definitions}  

To move from a topological triangulation to a geometric one, we need a way to impose a hyperbolic structure on each tetrahedron where the face pairings are hyperbolic isometries. A start is to assign angles to the edges of the triangulation.

\begin{defn}\label{def:anglestructure}
An \emph{angle structure} on an ideal triangulation $\tau$ of a 3-manifold is an assignment of angles to the edges of the tetrahedra in~$\tau$ satisfying the following conditions:
\begin{enumerate}
\item Angles assigned to opposite edges of a tetrahedron are equal, meaning it is enough to specify three angles $\theta_{3i-2}$, $\theta_{3i-1}$, and $\theta_{3i}$, for each tetrahedron~$T_i\in\tau$;
\item $\theta_{3i-2}+\theta_{3i-1}+\theta_{3i} = \pi$ for all~$i$; and
\item The sum of the angles surrounding an edge of the 3-manifold equals $2\pi$.
\end{enumerate}
\end{defn}

This paper utilizes both \emph{taut angle structures}, where $\theta_k\in\{0,\pi\}$ for all~$k$, and \emph{positive angle structures}, where $\theta_k>0$ for all~$k$.  Note that Condition~(2) of the definition implies that a positive angle structure must have all angles strictly between~0 and~$\pi$.  In this case, the three angles that sum to~$\pi$ specify a \emph{positively oriented ideal hyperbolic tetrahedron}.  The tetrahedron is isometric to one whose vertices are at $0$, $1$, $\infty$, and a complex number~$z$ with $\text{Im}(z)>0$, where $0$, $1$, and $z$ form a Euclidean triangle having the assigned angles.  If $\text{Im}(z)=0$, the tetrahedron is said to be \emph{degenerate}, and, if $\text{Im}(z)<0$, \emph{negatively oriented}.  If an angle structure~$\theta$ assigns the same angles to two tetrahedra, $T_1$ and $T_2$, then we will say they are \emph{isometric with respect to $\theta$} and write $T_1\cong_\theta T_2$.  

We use $\mathcal A(\tau)$ to denote the space of all positive angle structures on a triangulation~$\tau$ and note that, if $\tau$ has $n$ tetrahedra, $\mathcal A(\tau)\subset(0,\pi)^{3n}$ is a convex polytope with compact closure.  As such, the points obtained through coordinate-by-coordinate averaging of elements of $\mathcal A(\tau)$ are again in $\mathcal A(\tau)$.  Also, real-valued functions defined on $\overline{\mathcal A(\tau)}$ attain their extreme values. See \cite[Proposition~3.2]{AngledtoHyperbolic} for more details about the space of angle structures.

As will be described in~\ref{sec:CR}, the initial step in the Casson-Rivin program is to argue that $\mathcal A(\tau)$ is nonempty, i.e., that the triangulation admits a positive angle structure.  We accomplish this in our case by first showing (in~\ref{sec:veering}) that~$\tau_{p}$ admits a special type of taut angle structure --- a veering angle structure.   The notion of \lq\lq veering,'' first introduced by Agol~\cite{AgolVeering}, has several equivalent formulations.  The simplest to apply to our triangulations is from~\cite{HodgsonEtal}.

\begin{defn}\label{def:veering}
A \emph{veering angle structure} on an ideal triangulation $\tau$ of an oriented 3-manifold $M$ is a taut angle structure meeting an additional condition:  The edges of $M$ can be colored red and blue in such a way that within each tetrahedron
\begin{itemize}
    \item two of the edges with angle~0 are red and two are blue (the color on an angle~$\pi$ edge can be either red or blue)
    \item when viewed from any of the four ideal vertices, the $\pi$-angled edge is followed in the counterclockwise direction by a blue edge and then a red edge.    
 \end{itemize}
Each tetrahedron in a veering triangulation appears as the one in Figure~\ref{fig:veering}, where the $\pi$-angled edges on the diagonal can be red or blue.  If there is also a consistent \lq\lq upward'' orientation on the four faces (for example out of the page in Figure~\ref{fig:veering}), the angle structure is \emph{transverse taut}.  A triangulation that admits a veering angle structure is said to be \emph{veering}.
\end{defn} 
\begin{figure}[ht]
\includegraphics{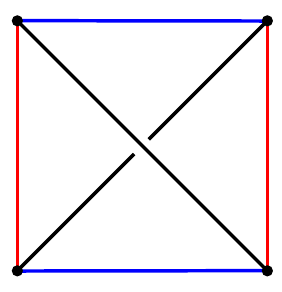}
\caption{A tetrahedron in a veering triangulation.\label{fig:veering}}
\end{figure}

 Positive angle structures give geometric structures on the tetrahedra in a triangulation.  In turn, these structures determine the face-pairing isometries.  Condition~(3) of Definition~\ref{def:anglestructure} guarantees that, under these isometries, the tetrahedra fit together to fill the space around the edges of the manifold, because the sum of dihedral angles around any edge is~$2\pi$.  However, there is no guarantee that the result is metrically complete.  In fact, if a 3-manifold has a complete hyperbolic structure of finite volume, Mostow-Prasad rigidity ensures that the structure is unique up to isometry \cite{MostowRigidity} \cite{PrasadRigidity}, so one should not expect completeness for an arbitrary point in~$\mathcal A(\tau)$.  It turns out that the independent work of Casson and Rivin~\cite{Rivin}, which forms the basis of the Casson-Rivin program in~\cite{AngledtoHyperbolic} and outlined in~\ref{sec:CR}, connects the question of completeness to the volume associated to an angle structure in~$\mathcal A(\tau)$, i.e., the sum of the volumes of the corresponding hyperbolic tetrahedra.
 
 \subsection{Volumes of angle structures} \label{sec:volume}

Recall (\cite{MilnorHypGeom}, for example) that the volume of an ideal hyperbolic tetrahedron, $T$, with dihedral angles $\theta_1$, $\theta_2$, and $\theta_3$ is given by
$$
V(T)=\lob(\theta_1)+\lob(\theta_2)+\lob(\theta_3),
$$
where the Lobachevsky function, $\lob$, is defined as
$$
\lob(x)=-\int_0^x\log|2\sin t|\,dt.
$$  
\begin{defn}\label{def:volume}
For a point $\theta\in\mathcal A(\tau)$, which specifies angles for the $n$ tetrahedra $T_i\in\tau$ and has all $\theta_i\in(0,\pi)$, define the \emph{volume associated to $\theta$}, $\mathcal{V}(\theta)$, by
$$
\mathcal V(\theta) =
\displaystyle{\sum_{i=1}^{n}V(T_i)}= \displaystyle{\sum_{i=1}^{3n}\lob(\theta_i)} 
$$
 and note that $\mathcal{V}(\theta)>0$ for $\theta\in\mathcal A(\tau)$.
\end{defn}

As summarized in~\cite{AngledtoHyperbolic}, $\lob$ is well-defined and continuous on $\R$.  Thus, the definition of~$\mathcal V$ can be extended to the closure of~$\mathcal A(\tau)$, which may include points with coordinates equal to 0 or $\pi$ where the integral itself is improper. If $\theta\in\overline{\mathcal A(\tau)}$, then $\mathcal{V}(\theta)\ge0$.

Straightforward tetrahedra-by-tetrahedra calculations show that for a positive angle structure $\theta$ and a tangent vector $\vec v\in T_\theta\mathcal{A}(\tau)$,
\begin{equation}
    \frac{\partial\mathcal{V}}{\partial \vec v}=\displaystyle{\sum_{i=1}^{3n} -v_i\log \sin \theta_i} \quad\text{ and }\quad\frac{\partial^2\mathcal{V}}{\partial \vec v\, ^2}<0.\label{eq:derivatives}
\end{equation}
See, for example, \cite[Lemma 5.3]{AngledtoHyperbolic}.

Because $\mathcal{V}$ is strictly concave down on~$\mathcal A(\tau)$, any critical point in $\mathcal{A}(\tau)$ is unique and an absolute maximum.  In fact, the first and second partials~\eqref{eq:derivatives} are enough to prove the same for a maximal point on the boundary. 

\begin{prop}\label{prop:maxunique}
Whenever $\mathcal{A}(\tau)\ne\emptyset$, the point at which the volume functional $\mathcal V: \overline{\mathcal A(\tau)}\rightarrow[0,\infty)$ attains its maximum is unique.
\end{prop}

\begin{proof}
If there is a critical point $\alpha$ in $\mathcal A(\tau)$, the strict concavity of $\mathcal{V}$ guarantees that $\alpha$ is unique and that $\mathcal{V}(\alpha)$ is the maximum volume.  If, on the other hand, there is no critical point on the interior, the compactness of $\overline{\mathcal A(\tau)}$ ensures that there is a point $\alpha\in\overline{\mathcal A(\tau)}-\mathcal A(\tau)$ with $\mathcal{V}(\alpha)$ maximal.  To show such an angle structure is unique, we first examine its coordinates.  Because $\alpha\notin\mathcal A(\tau)$, there is at least one tetrahedron in $\tau$ with an assigned angle of $0$ or $\pi$.  The maximality of $\alpha$ further constrains the angles of such a tetrahedron as noted in the following lemma.

\begin{lem}\label{lem:00pi}
When $\mathcal{V}(\alpha)$ is a maximum, any tetrahedron assigned an angle of $0$ or $\pi$ must have angles of $0$, $0$, and $\pi$ in some order. 
\end{lem}

\begin{proof}[Proof of Lemma~\ref{lem:00pi}]
By Condition~(2) of Definition~\ref{def:anglestructure}, the angles in a tetrahedron must sum to~$\pi$, so this statement is obviously true if one angle in a tetrahedron is assigned~$\pi$ or if two angles are assigned $0$. We now argue, as also noted in \cite[Proposition~7.1]{GwithF2Bridge} and~\cite[Proposition~7]{DehnFillings}, that  assigning exactly one $0$ angle contradicts the maximality of~$\alpha$. 

Consider, for example, a tetrahedron, $T_0\in\tau$, with nonnegative angles $0$, $\theta_0$, and $\pi-\theta_0$. For any $0<\lambda<1$, define a family of tetrahedra $T(t)|_{t\ge0}$ with angles $\theta_1=t$, $\theta_2=\theta_0-\lambda t$, and $\theta_3=\pi-(t+\theta_0-\lambda t)$, so $T(0)=T_0$ and, as $t$ increases, the angles of the tetrahedron change --- $\theta_1$ is increased, while $\theta_2$ and $\theta_3$ are decreased.  The first partial in~\eqref{eq:derivatives} restricted to this family of tetrahedra yields
$$
\frac{dV(T(t))}{dt}\Big|_{t=0^+}=\displaystyle{\sum_{i=1}^{3} -\frac{d\theta_i}{dt} \log \sin \theta_i}\Big|_{t=0^+}=+\infty.
$$
In other words, increasing the solitary $0$ angle by a small amount would increase $T_0$'s volume (and thus the total volume $\mathcal{V}$) by a much larger amount, violating the maximality of~$\alpha$.  Thus, if a maximal angle structure assigns any angles in $\{0,\pi\}$ to a tetrahedron, two of its angles are assigned $0$ and one~$\pi$, thus proving the lemma.  We call a tetrahedron with angles in $\{0,\pi\}$ \emph{flat} and observe that its volume is~$0$.
\end{proof}

Returning to the proof of Proposition~\ref{prop:maxunique}, we will use the constraints Lemma~\ref{lem:00pi} imposes on the coordinates of a maximal point $\alpha\in\overline{\mathcal A(\tau)}-\mathcal A(\tau)$ to prove that $\alpha$ is unique. Let $\beta$ be in $\overline{\mathcal A(\tau)}-\mathcal A(\tau)$ with $\mathcal{V}(\beta)=\mathcal{V}(\alpha)$ and let $\mu\in[0,\pi]^{3n}$ be the point whose coordinates are determined by averaging those of $\alpha$ and~$\beta$. Then $\mu\in\overline{\mathcal A(\tau)}$, because the set $\overline{\mathcal A(\tau)}$ is convex.  We will examine the volume of each tetrahedron $T$ under the angle structure assigned by~$\mu$ and will denote this volume by $\mathcal{V}(\mu|_T)$.

\begin{description}
    \item [Case 1] If $\alpha$ and $\beta$ assign the same angles to a tetrahedron $T$, then their average $\mu$ does as well, so $\mathcal{V}(\mu|_T)$ is equal to $\mathcal{V}(\alpha|_T)$ (and also $\mathcal{V}(\beta|_T)$).
    
    \item[Case 2] If $\alpha$ and $\beta$ assign different angles to a tetrahedron $T$ and both $\alpha|_T$ and $\beta|_T$ are flat, then, because $\alpha$ and $\beta$ are maximum points, the angles must be $0$, $0$, and $\pi$ (in different orders), and the coordinates of $\mu|_T$ are $0$, $\frac\pi2$, and $\frac\pi2$ in some order.  In this case, $\mathcal{V}(\mu|_T)$ equals $0$ just as $\mathcal{V}(\alpha|_T)$ and $\mathcal{V}(\beta|_{T})$ do.
    
    \item[Case 3] If $\alpha$ and $\beta$ assign different angles to a tetrahedron~$T$ and at least one of $\alpha|_T$ or $\beta|_T$ is not flat, we will argue that $\mathcal{V}(\mu|_T)$ is strictly greater than the average of $\mathcal{V}(\alpha|_T)$ and $\mathcal{V}(\beta|_T)$.  Without loss of generality suppose $\alpha|_T$ is not flat.  Define $\gamma(t)$ to be the line segment joining $\alpha|_T$ and~$\beta|_T$:
    $$
    \gamma(t)=t\alpha|_T + (1-t)\beta|_T \text{ for } t\in[0,1].
    $$
    Let $V(t)$ be the volume of the tetrahedron whose angles are given by $\gamma(t)$.  As a restriction of $\mathcal{V}$, the function $V$ is continuous on $[0,1]$.  Also, because $\alpha|_T$ is not flat, $\gamma$ is in $(0,\pi)^3$ for $0<t\le 1$, so $V$ is differentiable on $(0,1]$ and by the second partial in~\eqref{eq:derivatives}, $V^{\prime\prime}<0$.  Applying the Mean Value Theorem to $V$ yields points $c_1\in\left(0,\frac12\right)$ and $c_2\in\left(\frac12,1\right)$ with   
    \begin{align*}
        V^\prime(c_1)&=2\left(V\left(\textstyle{\frac12}\right)-V\left(0\right)\right)\\
        V^\prime(c_2)&=2\left(V\left(1\right)-V\left(\textstyle{\frac12}\right)\right)
    \end{align*}
    and negative second derivatives, so $V^\prime(c_1)>V^\prime(c_2)$, which means 
    \begin{equation*}
        2V\left(\textstyle{{\frac12}}\right)>V(1)+V(0).\footnote{If neither $\alpha|_T$ nor $\beta|_T$ is flat, this inequality follows immediately from the strict concavity of $V$.  We appeal to the Mean Value Theorem, because, in general, strict concavity does not extend to the boundary of $\mathcal A(\tau)$.}
    \end{equation*}
    Because $V\left(\textstyle{\frac12}\right)=\mathcal{V}(\mu|_T)$, $V(1)=\mathcal{V}(\alpha|_T)$, and $V(0)=\mathcal{V}(\beta|_T)$, this inequality implies $\mathcal{V}(\mu|_T)>\frac12(\mathcal{V}(\alpha|_T)+\mathcal{V}(\beta|_T))$.  
\end{description}

In all cases, the volume of a tetrahedron $T$ computed using $\mu$, the average of the angle structures, is greater than or equal to the average of the volumes of $T$ computed using the angle structures $\alpha$ and $\beta$, so summing the individual volumes over all tetrahedra in $\tau$ yields
$$
\mathcal{V}(\mu)\ge\textstyle{\frac12}(\mathcal{V}(\alpha)+\mathcal{V}(\beta))=\mathcal{V}(\alpha)
$$
with equality only when there are no tetrahedra in Case~3.  But $\mathcal{V}(\alpha)$ is the maximum volume, so $\mathcal{V}(\mu)\le\mathcal{V}(\alpha)$, meaning $\mathcal{V}(\mu)=\mathcal{V}(\alpha)$ and there are no Case~3 tetrahedra.  The
only way $\alpha$ and $\beta$ can assign different angles to a tetrahedron~$T$ is if $\alpha|_T$ and $\beta|_T$ are both flat (Case 2).  In this case,  $\mathcal{V}(\mu)=\mathcal{V}(\alpha)$, so $\mu$ is also maximal, but, as we have seen, the coordinates of $\mu|_T$ are $0$, $\frac\pi2$, and $\frac\pi2$, which is not allowed for maximal angle structures.  Therefore, $\alpha$ and $\beta$ cannot assign different angles to any of the tetrahedra and $\beta=\alpha$, which is the unique maximum point.
\end{proof}

Positive angle structures that maximize the total volume play an important role in the Casson-Rivin program.

\subsection{The Casson-Rivin program} \label{sec:CR}

In their proof of Casson and Rivin's theorem, Futer and Gu\'eritaud  relate  the critical point of the volume (if it exists) and the complete hyperbolic structure (if it exists).  More specifically, they show that the derivative of the volume vanishes in every direction exactly when Thurston's system of gluing equations guaranteeing completeness (given in \cite{Thurston1}) is satisfied.  In doing so, they prove their main result, a theorem based on independent work of Rivin and Casson usually cited as~\cite{Rivin}.

\begin{CRtheorem}[{\cite[Theorem 1.2]{AngledtoHyperbolic}}]  Let $M$ be an orientable 3-manifold with boundary consisting of tori, and let $\tau$ be an ideal triangulation of~$M$.  Then a point $\theta\in\mathcal A(\tau)$ corresponds to a complete hyperbolic metric on the interior of $M$ if and only if $\theta$ is a critical point of the functional $\mathcal V: \mathcal A(\tau)\rightarrow[0,\infty)$.
\end{CRtheorem}

In our situation, the manifold $M$ is the complement of an open neighborhood of the braid closure, $L_p$.  We seek a complete hyperbolic metric on $X_{p}$, which is homeomorphic to the interior of~$M$ and is triangulated by $\tau_{p}$. 

The Casson-Rivin Theorem will enable us to conclude $\tau_{p}$ is geometric as follows.  In~\ref{sec:veering}, we establish  that the space of positive angle structures on our triangulations, $\mathcal{A}(\tau_{p})$, is nonempty.  Consequently, the volume functional $\mathcal{V}:\overline{\mathcal{A}(\tau_{p})}\rightarrow[0,\infty)$ will attain its maximum.  Then, in Section~\ref{sec:proofq=1}, we show that the maximum point of $\mathcal{V}$ must belong to $\mathcal{A}(\tau_{p})$.  Then, by the Casson-Rivin Theorem, there is a positive angle structure on our triangulation $\tau_{p}$ that yields the complete hyperbolic structure on the link complement, $X_{p}$.  Thus, the constructed triangulations are geometric and $C^2\sigma_1^p\sigma_2^{-1}$ is hyperbolic for $p\ge1$.

\subsection{Symmetries of angle structures} \label{sec:symmetries}

When arguing that the maximum of the volume cannot appear on the boundary of $\mathcal{A}(\tau_{p})$, a certain symmetry of the space of positive angle structures will be helpful.

\begin{defn}\label{def:anglespacesymmetry}
Any symmetry, $\rho$, of a triangulation, $\tau$, induces a map on the space of positive angle structures by assigning the angles of $T$ to $\rho(T)$ for all $T\in\tau$.  When such an assignment results in a positive angle structure for all points in $\mathcal A(\tau)$, we say \emph{$\rho$ induces a symmetry of the space of positive angle structures}, which we will also denote by~$\rho$.  More specifically, given a symmetry $\rho:\tau\rightarrow\tau$ with $\rho(T_i)=T_j$, let $\theta\in\mathcal A(\tau)$ and define the coordinates of $\rho\theta$ by $\rho\theta_{3j-2}=\theta_{3i-2}$, $\rho\theta_{3j-1}=\theta_{3i-1}$, and $\rho\theta_{3j}=\theta_{3i}$. 
The angles of $\rho\theta$ are positive and meet Conditions~(1) and~(2) of Definition~\ref{def:anglestructure}, so $\rho$ induces a symmetry of~$\mathcal A(\tau)$ whenever Condition~(3) is also met.

\end{defn}

\begin{observation}\label{obs:equalvolume}
A symmetry $\rho:\mathcal A(\tau)\rightarrow\mathcal A(\tau)$ rearranges the angles of $\theta$ triple by triple, so $\mathcal{V}(\theta)= \mathcal{V}(\rho\theta)$.
\end{observation}

\section{The space of positive angle structures is nonempty}  \label{sec:anglestructuresnonempty}

We now return to the triangulations, $\tau_{p}$, defined in~\ref{sec:anyp}.  Our immediate goal is to complete the first step of the Casson-Rivin program by showing that $\mathcal A(\tau_{p})$ is nonempty, i.e., that $\tau_{p}$ admits a positive angle structure.  We will do so by showing that  $\tau_{p}$ has a veering angle structure, which can be deformed to a positive angle structure~\cite{HodgsonEtal}~\cite{veering}.  We then use the veering structure to describe an explicit coordinate system for $\mathcal A(\tau_{p})$ and conclude the section by writing the edge equations in these coordinates and using them to describe a useful symmetry.

\subsection{The triangulations are veering}\label{sec:veering}

While the initial triangulations $\hat{\tau}_p$ were not veering, the Pachner moves eliminated the obstructions and 
\begin{prop}\label{prop:veering}
Each triangulation $\tau_{p}$ is veering.
\end{prop}
\begin{proof}
With the exception of $s$, the tetrahedra in $\tau_{p}$ have an upward orientation induced by the product region containing the 6-braid, which is how the non-$s$ tetrahedra of $\tau_{p}$ are flattened in Figure~\ref{fig:anyp}.   Flatten $s$ as indicated in the same figure.  In each tetrahedron, assign an angle of~$\pi$ to the diagonal edges and an angle of~0 to the others.  Use red to color the 01 and 25 edges of each $w_i$ and~$w_i^\prime$; the edges $s(bc)$ and $s(ad)$; and the top diagonal of $w_0^\prime$ along with the bottom diagonal of~$w_p$.  (These are exactly the edges colored red in Figure~\ref{fig:anyp}.)  Color the remaining edges blue.  We claim this coloring forms a veering angle structure on~$\tau_{p}$, which is also transverse taut.  

By Observation~\ref{obs:veeringred}, the edges in the equivalence class $[w_0^\prime(01)]$ are exactly the edges colored red.  Thus, the blue edges compose the remaining equivalence classes, and the colorings are consistent with the face pairings that form~$M$. 

Next we show that the assignment of angles forms a taut angle structure.  The angle assignments themselves guarantee  Conditions~(1) and~(2) of  Definition~\ref{def:anglestructure} as well as the requirement that taut angle structures have angles in $\{0,\pi\}$.  Only diagonals are assigned an angle of~$\pi$, so, if there are exactly two diagonals in each equivalence class, Condition~(3) will also be met (the angle sum around each edge is~$2\pi$).  

The class $[w_0^\prime(01)]$, whose members are listed in Observation~\ref{obs:veeringred}, contains exactly two diagonals,~$w_0^\prime(02)$ and~$w_p(02)$.  The degree~5 edges of $X_{p}$ are completely described in Observation~\ref{obs:veeringblue5}.  The edges $[w_0^\prime(05)]$ and $[w_0^\prime(12)]$ each contain exactly two diagonals:  The top diagonal $w_1(15)$ is matched with the bottom diagonal $w_{p-1}^\prime(51)$  and similarly for $w_1^\prime(02)$ and $w_{p-1}(20)$ (or $s(db)$ and $s(ac)$ when $p=1$, which completes this case). 

To check the remaining $2p-2$ edges of~$X_{p}$ for $p>1$, i.e., those of degree~4, we examine $w_{i-1}(02)$ and $w_{i-1}^\prime(15)$, the bottom diagonals of $w_{i-1}$ and $w_{i-1}^\prime$ in the braid complement in Figure~\ref{fig:anyp}.  The edge $w_{i-1}(02)$ bounds both $w_{i-1}(012)$ and $w_{i-1}(025)$.  After passing through the $\sigma_1$~half-twist, the first of these faces is rotated to the back and identified with $w_{i}^\prime(102)$, whereas the second stays in front and is identified with $w_{i}(125)$.  These faces share the edge 12, which, after another $\sigma_1$ is twisted to the back and identified to $w_{i+1}^\prime(02)$.  A similar argument applies to~$w_{i-1}^\prime(15)$.  Symbolically,
\begin{align} 
\begin{split}\label{eq:deg4}
w_{i-1}(02)\sim w_{i}(12)= w_{i}^\prime(12)\sim w_{i+1}^\prime(02)\quad  &\text{ for } i=2, \dots p-2.\\
w_{i-1}^\prime(15)\sim w_{i}(05)=w_{i}^\prime(05)\sim w_{i+1}(15)\quad	&\text{ for } i=1, \dots p-1.
\end{split}
\end{align}
The only diagonals are the 02 and 15 edges, so the edge classes, $[w_i(12)]$ and $[w_i(05)]$, contain exactly two~$\pi$ angles as required.  The remaining  degree-4 edges contain the diagonals of~$s$, $s(ac)$ and $s(bd)$.  The identifications in Table~\ref{table:anyp} show that their equivalence classes contain exactly one other diagonal:
\begin{align} 
\begin{split}\label{eq:s}
s(ac)	\sim w_1^\prime(21) &= w_1(21) 	\sim w_2^\prime(20) \\ 
s(bd)	\sim w_{p-1}^\prime(21) &= w_{p-1}(21) 	\sim w_{p-2}(20).
\end{split}
\end{align}

With exactly two diagonals in each equivalence class, the angle structure is taut.  Each tetrahedron appears as in Figure~\ref{fig:veering}, and the gluings of Table~\ref{table:anyp} identify bottom faces to top faces, so the \lq\lq upward'' orientations are consistent, forming a transverse taut angle structure, which is also layered.
\end{proof}

\begin{cor}\label{cor:positiveanglestructure}
The triangulation $\tau_{p}$ admits a positive angle structure.
\end{cor}
\begin{proof}
In their main result, \cite[Theorem 1.5]{HodgsonEtal}, Hodgson, Rubinstein, Segerman, and Tillmann prove that veering triangulations admit positive angle structures (which they call strict angle structures).  A constructive proof showing how to deform a veering angle structure to a positive angle structure has also been given by Futer and Gu\'eritaud~\cite[Theorem 1.3]{veering}.
\end{proof}

\subsection{Coordinates for the space of positive angle  structures}\label{sec:coordinates}

The veering structure on~$\tau_{p}$ allows us to introduce convenient coordinates for a point $\theta$ in $\mathcal A(\tau_{p})$.  For $T$ in~$\tau_{p}$, let~$\theta_D T$ denote the angle assigned to the diagonals of~$T$, $\theta_R T$ the angle assigned to the red edges, and $\theta_B T$ to the blue edges.  Thus, we can write $\theta\in\mathcal{A}(\tau_{p})\subset(0,\pi)^{3(2p+1)}$ as
\begin{multline*}
(\theta_R w_0^\prime,
\theta_B w_0^\prime,
\theta_D w_0^\prime,\dots,
\theta_R w_{p-1}^\prime,
\theta_B w_{p-1}^\prime,
\theta_D w_{p-i}^\prime,\\
\theta_R w_1,
\theta_B w_1,
\theta_D w_1,\dots,
\theta_R w_p,
\theta_B w_p,
\theta_D w_p,
\theta_R s,
\theta_B s,
\theta_D s).
\end{multline*}
\begin{observation}\label{obs:edgeeqs}
Using this notation, the edge identifications in Observations~\ref{obs:veeringred} and~\ref{obs:veeringblue5} and Equations~\eqref{eq:deg4} and~\eqref{eq:s} together with Condition~(3) of Definition~\ref{def:anglestructure} yield the following \emph{edge equations}:
\begin{align*}
\left(\sum_{i=0}^{p-1} 2\theta_R w_i^\prime\right) +\left(\sum_{i=1}^{p} 2\theta_R w_i\right) +2\theta_R s + \theta_Dw_0^\prime+\theta_Dw_p&=2\pi\\
\theta_B w_0^\prime+\theta_D w_1+\theta_B s+\theta_D w_{p-1}^\prime+\theta_B w_p&=2\pi\\
\theta_B w_0^\prime+\theta_D w_1^\prime+\theta_B s+\theta_D  w_{p-1}+\theta_B w_p&=2\pi\\
\theta_D w_{i-1}+\theta_B w_{i}+\theta_B w_{i}^\prime+\theta_D  w_{i+1}^\prime&=2\pi&\text{ for } i=2, \dots p-2\\
\theta_D w_{i-1}^\prime+\theta_B w_{i}+\theta_B w_{i}^\prime+\theta_D  w_{i+1}&=2\pi&\text{ for } i=1, \dots p-1\\
\theta_D s+\theta_B w_1^\prime+\theta_B w_1+\theta_D  w_2^\prime&=2\pi\\
\theta_D s+\theta_B w_{p-1}^\prime+\theta_B w_{p-1}+\theta_D  w_{p-2}&=2\pi.
\end{align*}
These equations hold for $p>1$ as long as each term is defined.  (Recall that there are no tetrahedra labelled $w_0$ or $w_p^\prime$.)  

The first two equations also hold when $p=1$, and, after substituting $s(db)$ and~$s(ac)$ for the undefined terms $w_1^\prime(02)$ and $w_{p-1}(20)$ as in Observation~\ref{obs:veeringblue5}, so does the third:
$$
\theta_B w_0^\prime+\theta_D s+\theta_B s+\theta_D  s+\theta_B w_1=2\pi.
$$
\end{observation}

\begin{observation}\label{obs:diagonaleqs}
Several pairs of  edge equations above share angles coming from the blue edges.  These angles will cancel when one equation is subtracted from the other, yielding the following equalities for the diagonals:
\begin{align*}
    \theta_D w_1-\theta_D w_1^\prime
    =\theta_D w_3-\theta_D w_3^\prime
    =\theta_D w_5-\theta_D w_5^\prime
    =&\cdots
    \\
\theta_D w_2-\theta_D w_2^\prime
    =\theta_D w_4-\theta_D w_4^\prime
    =\theta_D w_6-\theta_D w_6^\prime
    =&\cdots.
\end{align*}
In addition,
\begin{align*}
      \theta_D w_{p-2}-\theta_D w_{p-2}^\prime
    &=\theta_D w_p-\theta_D s
      \\
       \theta_D s-\theta_D w_0^\prime
    &=\theta_D w_2-\theta_D w_2^\prime
    \\
 \theta_D w_{p-1}-\theta_D w_{p-1}^\prime
    &=\theta_D w_1-\theta_D w_1^\prime.
\end{align*}
These equations hold when $p>1$ and each term is defined.
The final equation, which is derived by taking the difference between the equations for the degree-5 edges implies that, if $p$ is odd (so $p-1$
 is even), then all of the differences listed above are equal.  If $p$ is even, all differences, $\theta_D w_j-\theta_D w_j^\prime$, with $j$ even are also equal to $\theta_D s-\theta_D w_0^\prime$ and $\theta_D w_p-\theta_D s$, but not necessarily to $\theta_D w_1-\theta_D w_1^\prime$, etc.

\end{observation}
The corresponding equation for $p=1$ is:
$$
      \theta_D s-\theta_D w_0^\prime
    =\theta_D w_1-\theta_D s.
$$

\subsection{A symmetry of the space of angle structures}\label{sec:symmetry}

Consider the closure of $C^2\sigma_1^p\sigma_2^{-1}$ as shown on the right of Figure~\ref{fig:6braid}.  Moving the lone half-twist up so that it occurs halfway along the $\sigma_1$ half-twists reveals an order 2 symmetry of $X_p$ --- rotate about a horizontal line through the lone half-twist.  This involution induces a symmetry on the triangulation $\tau_p$ --- rotate the tetrahedra in Figure~\ref{fig:anyp} about a horizontal line and take $s$ to itself.  This symmetry will respect  angle structures:

\begin{prop} \label{prop:symmetry}
An involution $\iota:\tau_{p}\rightarrow\tau_{p}$ that fixes $s$ and takes $w_i^\prime$ to $w_{p-i}$ by matching up their $R$, $B$, and $D$ edges induces a symmetry of $\mathcal A(\tau_{p})$.
\end{prop}

\begin{proof}
Let $\theta\in\mathcal A(\tau_{p})$.  According to Definition~\ref{def:anglespacesymmetry}, to verify that $\iota\theta$ is also in $\mathcal A(\tau_{p})$, we need to check that the sum of the dihedral angles at each edge of of~$X_{p}$ is~$2\pi$, when the angle measure is given by~$\iota\theta$.  Thus, it is enough to verify that $\iota\theta$ satisfies the edge equations listed in Observation~\ref{obs:edgeeqs} where
$$
\iota\theta_\star s=\theta_\star s, \quad
\iota\theta_\star w_i^\prime=\theta_\star w_{p-i},\quad \text{and}\quad 
\iota\theta_\star w_i=\theta_\star w_{p-i}^\prime
$$
for $\star=R, B$, and $D$.

The first equation describes the angle sum around the red edge of $X_{p}$, which consists of all $R$ edges and the red diagonals of $w_0^\prime$ and~$w_p$.  The involution $\iota$ permutes these edges, thereby permuting the terms of the sum:
\begin{align*}
\left(\sum_{i=0}^{p-1} 2\iota\theta_R w_i^\prime\right) +\left(\sum_{i=1}^{p} 2\iota\theta_R w_i\right) +2\iota\theta_R s + \iota\theta_Dw_0^\prime+\iota\theta_Dw_p&=\\
\left(\sum_{i=1}^{p} 2\theta_R w_i\right) +\left(\sum_{i=0}^{p-1} 2\theta_R w_i^\prime\right) +2\theta_R s + \theta_Dw_p+\theta_Dw_0^\prime&=2\pi,\\
\end{align*}
so the first edge equation is satisfied by $\iota\theta$.  The terms in the next two sums, those for the degree-5 edges, are also permuted by~$\iota$, so those equations are satisfied by $\iota\theta$ as well.  

Something a little different happens with the degree-4 edges (only present when $p>1$).  To confirm the final equation of Observation~\ref{obs:edgeeqs}, for example, we need to examine
$$
\iota\theta_D s+\iota\theta_B w_{p-1}^\prime+\iota\theta_B w_{p-1}+\iota\theta_Dw_{p-2},
$$
but this sum is just
$$
\theta_D s+\theta_B w_1+\theta_B w_1^\prime+\theta_D  w_2^\prime,
$$
which equals $2\pi$ by the penultimate equation in Observation~\ref{obs:edgeeqs}.  In this instance, the involution takes the terms in one sum to the terms in another, effectively permuting the edge equations themselves.  The same happens for the remaining degree-4 equations.  
\end{proof}

\begin{cor}\label{cor:symmetryedgeeqs}
Let $\iota:\tau_{p}\rightarrow\tau_{p}$ be as in Proposition~\ref{prop:symmetry} and let $k$ be such that $p=2k$ if $p$ is even and $p=2k+1$ if $p$ is odd.  Positive angle structures $\theta\in\mathcal A(\tau_{p})$ with the property that $\iota\theta=\theta$ have $w_i^\prime\cong_\theta w_{p-i}$ and thus satisfy a simpler list of edge equations: 
\begin{align}
\left(\sum_{i=1}^{p} 4\theta_R w_i\right) +2\theta_R s + 2\theta_Dw_p&=2\pi\label{eq:deg4p+4}\\
2\theta_B w_p+2\theta_D w_1+\theta_B s&=2\pi\label{eq:deg5a}\\
2\theta_B w_p+2\theta_D w_{p-1}+\theta_B s&=2\pi\label{eq:deg5b}\\
\theta_D w_{i-1}+\theta_B w_{i}+\theta_B w_{p-i}+\theta_D  w_{p-(i+1)}&=2\pi&\text{ for } i=2, \dots k\label{eq:deg4a}\\
\theta_D w_{p-(i-1)}+\theta_B w_{i}+\theta_B w_{p-i}+\theta_D  w_{i+1}&=2\pi&\text{ for } i=1, \dots k\label{eq:deg4b}\\
\theta_D s+\theta_B w_{p-1}+\theta_B w_1+\theta_D  w_{p-2}&=2\pi.\label{eq:deg4c}
\end{align}
In addition,
\begin{align}
 \theta_D w_p=\theta_D s\quad\text{ and }\quad\theta_D w_i=\theta_D w_{p-i} \quad\text{ for } i=1, \dots k.\label{eq:diagonals}
\end{align}
These equations hold for all $p\ge1$, with one exception.  When $p=1$, Equation~\eqref{eq:deg5b} reads:
$
2\theta_B w_1+2\theta_D s+\theta_B s=2\pi.
$
\end{cor}

\begin{proof}
By assumption,
$\iota\theta_\star w_i^\prime=\theta_\star w_i^\prime$ for $\star=R, B$, and $D$, and, by definition, $\iota\theta_\star w_i^\prime$ also equals $\theta_\star w_{p-i}$, so $\theta_\star w_i^\prime=\theta_\star w_{p-i}$, and the tetrahedra $w_i^\prime$ and $w_{p-i}$ have the same angle assignments under~$\theta$, so $w_i^\prime\cong_\theta w_{p-i}$.  This isometry allows us to obtain Equations~\eqref{eq:deg4p+4}--\eqref{eq:deg4c}, by replacing $w_i^\prime$ with $w_{p-i}$ in the edge equations of Observation~\ref{obs:edgeeqs} and eliminating redundancies.  

To obtain Equation~\eqref{eq:diagonals} for $p=1$, replace $w_0^\prime$ with $w_1$ in the $p=1$ equation of Observation~\ref{obs:diagonaleqs}, which shows 
$$
\theta_D s-\theta_D w_1
    =\theta_D w_1-\theta_D s,
$$
so $\theta_D w_1=\theta_D s$.  
If $p>1$, Equations~\eqref{eq:deg5a} and~\eqref{eq:deg5b} imply that $\theta_D w_1$ equals $\theta_D w_{p-1}$, which, because of the isometry, equals $\theta_D w_1^\prime$, so the difference $\theta_D w_1-\theta_D w_1^\prime$ equals~0 as do all odd-index differences listed in Observation~\ref{obs:diagonaleqs}. Replacing $w_i^\prime$ with $w_{p-i}$, yields Equation~\eqref{eq:diagonals} for odd~$i$.  The remaining equations are derived differently depending on the parity of~$p$.  As noted in Observation~\ref{obs:diagonaleqs}, when $p$ is odd, all differences are equal (in this case, equal to~0), so, after replacing $w_i^\prime$ with $w_{p-i}$, Equation~\eqref{eq:diagonals} holds for all $i$ when $p$ is odd. To see that Equation~\eqref{eq:diagonals} also holds when $p$ and $i$ are even, note that
$$
\theta_D w_p-\theta_D s
   =\theta_D w_2-\theta_D w_2^\prime
    =\theta_D w_{p-2}-\theta_D w_{p-2}^\prime
    =\theta_D s-\theta_D w_0^\prime
    =\theta_D s-\theta_D w_p
$$
so $\theta_D s=\theta_D w_p$. Thus, all differences in Observation~\ref{obs:diagonaleqs} with even index are also~0 and  replacing  $w_i^\prime$ with $w_{p-i}$ results in Equation~\eqref{eq:diagonals}.
\end{proof}

\section{The maximal volume occurs at a positive angle structure}\label{sec:proofq=1}

Corollary~\ref{cor:positiveanglestructure} shows that the space of positive angle structures on each triangulation~$\tau_{p}$ is nonempty, so we have accomplished the first step of the Casson-Rivin program described by  Futer and Gu\'eritaud in~\cite{AngledtoHyperbolic} and summarized in~\ref{sec:CR}.  In this section, we complete the program, proving the main result.

\begin{thm}\label{thm:q=1}
Let $L_p$ be the closure of the 3-braid $C^2\sigma_1^p\sigma_2^{-1}$ and $X_{p}$ be its complement in the 3-sphere.  Then there is an ideal triangulation of $X_{p}$ that is geometric.
\end{thm}

\begin{proof}

Let $\tau_{p}$ be the triangulation defined in~\ref{sec:anyp}. By  Corollary~\ref{cor:positiveanglestructure}, $\mathcal{A}(\tau_{p})\ne\emptyset$, which means the volume functional $\mathcal{V}$ attains its maximum on the compact set $\overline{\mathcal{A}(\tau_{p})}$, say at the point~$\alpha$.  A maximal point has two important properties.  The first was proved in Lemma~\ref{lem:00pi}.

\begin{property}\label{property:1}
If $\alpha$ is maximal and assigns an angle in $\{0,\pi\}$ to a tetrahedron, then the tetrahedron must be flat, i.e., its angles are $0$, $0$, $\pi$ in some order.  
\end{property}

The second important property derives from a symmetry of the coordinates of a maximal point.  Let the involution $\iota:\tau_{p}\rightarrow\tau_{p}$ be as in Proposition~\ref{prop:symmetry}, so $\iota$ fixes~$s$ and takes $w_i^\prime$ to $w_{p-i}$.  Then $\iota$ induces a symmetry on $\overline{\mathcal{A}(\tau_{p})}$, which rearranges the coordinates of an angle structure tetrahedron by tetrahedron, so, just as in Observation~\ref{obs:equalvolume}, $\mathcal{V}(\iota\alpha)=\mathcal{V}(\alpha)$.  Thus, if $\alpha$ is maximal, so is $\iota\alpha$, but  maximal points are unique by Proposition~\ref{prop:maxunique}, so $\iota\alpha=\alpha$, and  Corollary~\ref{cor:symmetryedgeeqs} together with Definition~\ref{def:anglestructure} guarantee the following.

\begin{property}\label{property:2}
If $\alpha$ is maximal, then the angles assigned by $\alpha$ belong to $[0,\pi]$; sum to~$\pi$ within a tetrahedron; and are determined by the angles assigned to $w_1, w_2, \dots, w_p$ and $s$, which must satisfy Equations~\eqref{eq:deg4p+4}--\eqref{eq:diagonals}.
\end{property}

Our goal is to use these properties to show that the maximal point, $\alpha$, is in~$\mathcal{A}(\tau_{p})$, not its boundary.  Then we can apply the Casson-Rivin Theorem from~\ref{sec:CR} \cite[Theorem 1.2]{AngledtoHyperbolic} and conclude that $\alpha$ corresponds to the complete hyperbolic structure on $X_{p}$, thus proving Theorem~\ref{thm:q=1}.  We will show that $\alpha$ is in $\mathcal{A}(\tau_{p})$ by showing that points in $\mathcal{B}=\overline{\mathcal{A}(\tau_{p})}-\mathcal{A}(\tau_{p})$ will never maximize the volume.  We do so by proving the following.

\begin{prop}\label{prop:volume=0}
Any $\beta\in\mathcal{B}$ satisfying Properties~\ref{property:1} and~\ref{property:2} must have $\mathcal{V}(\beta)=0$.
\end{prop}

\begin{proof}[Proof of Proposition~\ref{prop:volume=0}]

Let $\beta\in\mathcal{B}$ satisfy both Property~\ref{property:1} and Property~\ref{property:2}.  Because $\beta$ is in the boundary of the space of positive angle structures, there is a tetrahedron in $\tau_{p}$ to which $\beta$ assigns  angles from the set $\{0, \pi\}$.  By Property~\ref{property:1}, this tetrahedron is flat (has angles $0$, $0$, and $\pi$).  Using Property~\ref{property:2}, we can conclude that $\beta$ must assign~$\pi$ to at least one of the edges of $w_1, w_2, \dots w_p$ or $s$.  We now explore which edges can have such assignments and determine the resulting volume.

Observe that, because all angle assignments are nonnegative, assigning an angle measure of~$\pi$ to any $R$ edge of a $w$ tetrahedron violates Equation~\eqref{eq:deg4p+4} of Corollary~\ref{cor:symmetryedgeeqs} and thus Property~\ref{property:2}.  Therefore, no $\beta_R w_i$ can equal $\pi$.  We will use this result often, so we record it as a lemma and follow with another useful lemma.

\begin{lem}\label{lem:noRsarepi}
If $\beta\in\mathcal{B}$ satisfies both Property~\ref{property:1} and~\ref{property:2}, $\beta_R w_i$ cannot equal $\pi$ for $i=1,\dots,p$.\qed
\end{lem}

\begin{lem}\label{lem:wporsflat}
If $\beta\in\mathcal{B}$ satisfies both Property~\ref{property:1} and~\ref{property:2}, and either $w_p$ or $s$ is flat, then all tetrahedra are flat, so $\mathcal{V}(\beta)=0$.
\end{lem}

\begin{proof}[Proof of Lemma~\ref{lem:wporsflat}]

If a tetrahedron is flat, one of its edges is assigned an angle of~$\pi$.  By Lemma~\ref{lem:noRsarepi}, $\beta_R w_p$ cannot equal~$\pi$.  However, $\beta_R s$ could equal~$\pi$, and, if this is the case, Equation~\eqref{eq:deg4p+4} forces $\beta_R w_i$ to equal~$0$ for all $i=1,\dots,p$, so, by Property~\ref{property:1}, the non-$s$ tetrahedra are also flat.  Because, $\beta_D w_p$ also appears in Equation~\eqref{eq:deg4p+4}, a similar argument shows that if $\beta_D w_p=\pi$, all tetrahedra are flat.  By Equation~\eqref{eq:diagonals}, $\beta_D w_p=\beta_D s$, so the same holds if $\beta_D s=\pi$.  Thus, it only remains to check what happens if the angle at a $B$ edge of $w_p$ or $s$ equals $\pi$.
\begin{itemize}
    \item If $\beta_B w_p=\pi$, Equation~\eqref{eq:deg5a} implies that $\beta_B s=0$, so, by Property~\ref{property:1}, there is a $\pi$ angle at either the $R$ or $D$ edge of $s$, and, in either case, as observed above, all tetrahedra are flat.  
    \item If $\beta_B s=\pi$, then $\beta_D s=0$, so $\beta_D w_p$ also equals~$0$ (Equation~\eqref{eq:diagonals}), and, by Property~\ref{property:1}, either $\beta_R w_p$ or $\beta_B w_p$ equals~$\pi$.  Because of Lemma~\ref{lem:noRsarepi}, $\beta_B w_p$  must be~$\pi$, so, by the previous case, all tetrahedra are flat.
\end{itemize}
Therefore, if either $w_p$ or $s$  has an  assigned angle of~$\pi$, $\mathcal{V}(\beta)=0$.
\end{proof}

Returning to the proof of Proposition~\ref{prop:volume=0}, recall that $\beta\in\mathcal{B}$ must assign~$\pi$ to at least one of the edges of $w_1, w_2, \dots w_p$ or~$s$.  Lemmas~\ref{lem:noRsarepi} and~\ref{lem:wporsflat} cover the $R$ edges and the tetrahedra $w_p$ and $s$, so the only edges left to consider are the $B$ and $D$ edges of $w_1, w_2, \dots w_{p-1}$.  We start with the $B$ edges and first examine the case when $\beta_B w_1=\pi$.

If $\beta_B w_1=\pi$, then $\beta_D w_1$ equals~$0$ and, by Equation~\eqref{eq:diagonals}$_{i=1}$, so does $\beta_D w_{p-1}$.  Property~\ref{property:1} implies that in $w_{p-1}$ one of the other angles must be $\pi$, but, because of  Lemma~\ref{lem:noRsarepi}, it cannot be the angle at the $R$ edge. Consequently, $\beta_B w_{p-1}$ equals~$\pi$ and Equation~\eqref{eq:deg4c} implies
$\beta_D s$ must equal~$0$, so $s$ is flat and, by Lemma~\ref{lem:wporsflat}, $\mathcal{V}(\beta)=0$.  

Having shown that if $\beta$ assigns the angle $\pi$ to a $B$ edge of $w_1$, then $\mathcal{V}(\beta)=0$, we now consider the $B$ edges of $w_j$ with $1<j\le k$.  By repeatedly applying Equations~\eqref{eq:diagonals} and~\eqref{eq:deg4a} together with Property~\ref{property:1} and Lemma~\ref{lem:noRsarepi}, we will push the flatness of $w_j$ all the way down to $w_1$, allowing us to once again conclude $\mathcal{V}(\beta)=0$.  In particular, if $\beta_B w_j=\pi$, then $\beta_D w_j$ equals~$0$ and, by Equation~\eqref{eq:diagonals}$_{i=j}$, so does $\beta_D w_{p-j}$.  There is a $\pi$ angle in $w_{p-j}$ (Property~\ref{property:1}) and it cannot occur at the $R$ edge (Lemma~\ref{lem:noRsarepi}), so $\beta_B w_{p-j}$, will equal~$\pi$.  With both $\beta_B w_j$ and $\beta_B w_{p-j}$ equal to $\pi$, Equation~\eqref{eq:deg4a}$_{i=j}$ implies that $\beta_D w_{j-1}=0$, so, by  Equation~\eqref{eq:diagonals}$_{i=j-1}$, $\beta_D w_{p-(j-1)}=0$ also.  Another application of Property~\ref{property:1} and Lemma~\ref{lem:noRsarepi} shows both $\beta_B w_{j-1}$ and $\beta_B w_{p-(j-1)}$ are equal to $\pi$ and we can apply Equation~\eqref{eq:deg4a}$_{i=j-1}$ to conclude $\beta_D w_{j-2}=0$.  Continuing in this manner and eventually applying Equation~\eqref{eq:deg4a}$_{i=2}$ when both $\beta_B w_2$ and $\beta_B w_{p-2}$ are equal to $\pi$, allows us to conclude that $\beta_D w_1=0$.  Thus (by Property~\ref{property:1} and Lemma~\ref{lem:noRsarepi}), $\beta_B w_1=\pi$, so $\mathcal{V}(\beta)=0$.

It remains to examine the $B$ edges of $w_j$ where $j>k$.  If $\beta_B w_j$ equals $\pi$, then $\beta_D w_j=0$, so applying Equation~\eqref{eq:diagonals}$_{i=p-j}$, Property~\ref{property:1}, and Lemma~\ref{lem:noRsarepi} shows that $\beta_B w_{p-j}=\pi$.  Because $p-j\le k$, this case has already been covered.  Therefore, if any of the $B$ edges of $w_1, w_2, \dots w_{p-1}$ have angle equal to~$\pi$, then $\mathcal{V}(\beta)=0$.

The proof for the $D$ edges is a little simpler, because there is no need to appeal to Lemma~\ref{lem:noRsarepi}. If $\beta_D w_1=\pi$, then, by Equation~\eqref{eq:diagonals}$_{i=1}$, $\beta_D w_{p-1}=\pi$, so both $\beta_B w_1$ and $\beta_B w_{p-1}$ are equal to~$0$. In this situation, applying Equation~\eqref{eq:deg4c} yields $\beta_D s=\pi$,
so $s$ is flat and, by Lemma~\ref{lem:wporsflat}, $\mathcal{V}(\beta)=0$.  If $j\le k$ and $\beta_D w_j$ equals $\pi$, then, by Equation~\eqref{eq:diagonals}$_{i=j}$, so does $\beta_D w_{p-j}$, and both $\beta_B(w_j)$ and $\beta_B(w_{p-j})$ equal~$0$, so Equation~\eqref{eq:deg4a}$_i$ implies $\beta_D w_{j-1}=\pi$.  Continuing in this manner shows that if $\beta_D w_j$ equals~$\pi$, so does $\beta_D w_1$, and thus $\mathcal{V}(\beta)=0$.  If $j> k$ and $\beta_D w_j=\pi$, then, by Equation~\eqref{eq:diagonals}$_{i=p-j}$, $\beta_D w_{p-j}=\pi$ with $p-j\le k$, a previous case.  
\end{proof}

Any point maximizing the volume functional must satisfy Properties~\ref{property:1} and~\ref{property:2}.  By proving Proposition~\ref{prop:volume=0}, we have shown that any point, $\beta$, on the boundary of $\mathcal{A}(\tau_{p})$ satisfying these properties has $\mathcal{V}(\beta)=0$.  Therefore, the maximal point $\alpha$ must be on the interior of $\mathcal A(\tau_{p})$, and, by the Casson-Rivin Theorem (\ref{sec:CR} and \cite[Theorem 1.2]{AngledtoHyperbolic}), $\alpha$ corresponds to the complete hyperbolic structure on~$X_{p}$.  Thus, the triangulation $\tau_{p}$ is geometric, which concludes the proof of Theorem~\ref{thm:q=1}.
\end{proof}

\begin{cor}\label{cor:notconjugate}
In the braid group, $C^2\sigma_1^p\sigma_2^{-1}$ is not conjugate to $\sigma_1^{p_0}\sigma_2^{q_0}$ for integers $p_0$ and $q_0$.
\end{cor}

\begin{proof}
Closures of braids of the form $\sigma_1^{p_0}\sigma_2^{q_0}$ are not hyperbolic, but Theorem~\ref{thm:q=1} shows that the closure of $C^2\sigma_1^p\sigma_2^{-1}$ is.
\end{proof}

\begin{cor}\label{cor:pretzelstoo}
The complements of the $(-2,3,n)$-pretzel knots and links admit geometric triangulations for $n\ge7$.
\end{cor}

\begin{proof}
 Recall that the braid group has the single relation $\sigma_1\sigma_2\sigma_1 = \sigma_2\sigma_1\sigma_2 $ and that $C$, which is this element's square, is central.
 Conjugate words in the braid group yield equivalent braid closures, so, using the relation and the centrality of~$C$, $L_p$, the closure of $C^2\sigma_1^p\sigma_2^{-1}$ with $p\ge1$, is also the closure of 
\begin{align*}
    (\sigma_2\sigma_1\sigma_2)(\sigma_1\sigma_2\sigma_1)C\sigma_1^p\sigma_2^{-1} &\sim \sigma_2\sigma_1\sigma_2 C\sigma_1^{p+2}\\
    &=\sigma_1\sigma_2C\sigma_1^{p+3}\\
    &\sim\sigma_2(\sigma_1\sigma_2\sigma_1)(\sigma_1\sigma_2\sigma_1)         \sigma_1^{p+4}\\
    &=\sigma_1\sigma_2\sigma_1\sigma_1\sigma_1\sigma_2\sigma_1^{p+5}\\
    &\sim\sigma_1^3\sigma_2\sigma_1^{p+6}\sigma_2;
\end{align*}
here $\sim$ denotes the equivalence relation of conjugacy.  

As indicated in Figure~\ref{fig:pretzel}, the closure of $\sigma_1^3\sigma_2\sigma_1^{p+6}\sigma_2$ is a pretzel. Using our definitions, $\sigma_1$ and  $\sigma_2$ generate left-handed (negative) twists, so $L_p$ is a pretzel with $-(p+6)$, $2$, and $-3$ half-twists, or equivalently  --- after rolling the $p+6$ half-twists to the right --- $L_p$ is the $(2,-3, -(p+6))$-pretzel, and $X_p$ is its complement. By Theorem~\ref{thm:q=1}, $X_p$ admits a geometric triangulation.  But the complement of the $(-2,3,p+6)$-pretzel is homeomorphic (via a reflection) to $X_p$, so it also has a geometric triangulation.
\end{proof}
\begin{figure}[ht]\vskip.25in
\begin{overpic}[unit=.25in]{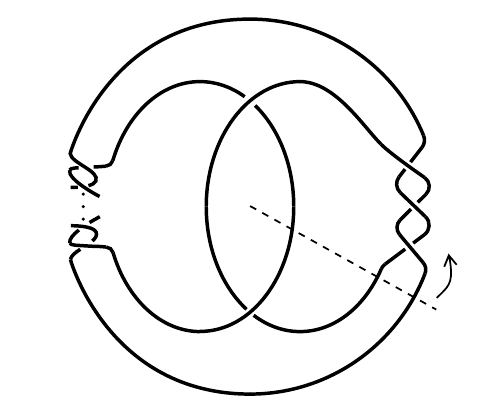}
\put(-.5,3.2){\small{$p+6$}}
\put(.7,3.2){$\Bigg\{$}
\end{overpic}
\caption{The $(-(p+6),2,-3)$-pretzel is also the closure of the braid $\sigma_1^3\sigma_2\sigma_1^{p+6}\sigma_2$.  Start at the dashed line and read the braid counterclockwise (strands numbered from the outside in).\label{fig:pretzel}}
\end{figure}

The braid relation $\sigma_1\sigma_2\sigma_1 = \sigma_2\sigma_1\sigma_2 $ and the centrality of $C=(\sigma_1\sigma_2)^{3}$ can also be used to show that $L_p$ is equivalent to the braid closure of $\sigma_1^{p+1}(\sigma_1\sigma_2)^{5}$:
$$
C^2\sigma_1^p\sigma_2^{-1}=\sigma_1^p(\sigma_1\sigma_2)^{6}\sigma_2^{-1}=\sigma_1^p(\sigma_1\sigma_2)^{5}\sigma_1\sim \sigma_1^{p+1}(\sigma_1\sigma_2)^{5}.
$$
In this form, $L_p$ is a T-link as defined by Birman and Kofman in~\cite{TLinks}, where they show that T-links are in one-to-one correspondence with Lorenz links.
\begin{cor}\label{cor:Tlinkstoo}
The complements of the T-links/Lorenz links formed as closures of the braids $\sigma_1^{p+1}(\sigma_1\sigma_2)^{5}$ admit  geometric triangulations for $p>1$. \qed
\end{cor}

\section{Extending the construction}\label{sec:extendingconstruction}

Having constructed geometric triangulations for the braid closure of $C^2\sigma_1^p\sigma_2^{-1}$, it is natural to ask whether this construction 
can be extended to cover more cases.  For example, hyperbolic T-links on three strands must be equivalent to closures of braids of the form $C^k\sigma_1^p\sigma_2^{-1}$ with $p>0$ and $k\ge2$, so adding more full twists to the construction would provide triangulations of all possible hyperbolic T-links on three strands.  However, a useful combinatorial triangulation is not obvious.

Another possible extension would be to add more $\sigma_2^{-1}$ half-twists.  The construction of $\hat{\tau}_p$ in Section~\ref{sec:anyp} can be extended, without difficulty, to $C^2\sigma_1^p\sigma_2^{-q}$ for all $q>0$, and, in each case, there is also a sequence of Pachner moves that results in a veering triangulation.  Thus, the first step of the Casson-Rivin program can be carried out.  However, the final push for geometricity (the analog of Proposition~\ref{prop:volume=0}) would require additional arguments.  After completing $C^2\sigma_1^p\sigma_2^{-q}$, it may be possible to further extend the construction to  $C^2\sigma_1^{p_1}\sigma_2^{-q_1}\cdots\sigma_1^{p_s}\sigma_2^{-q_s}$, moving closer to the general form in the Futer-Kalfagianni-Purcell characterization of hyperbolic 3-braids~\cite{FareyManifolds} stated in Section~\ref{sec:intro}.

\appendix

\section{Visualizing triangulations of 2-bridge link complements}\label{app:seeing2-bridge}

Sakuma and Weeks constructed topological triangulations of 2-bridge link complements \cite{SakumaWeeks}, which  Futer showed were geometric \cite[Appendix]{GwithF2Bridge}.  Futer's approach follows that of Gu\'eritaud, whose development of geometric triangulations of once-punctured torus bundles and 4-punctured sphere bundles forms the bulk of~\cite{GwithF2Bridge}. This appendix presents another way to visualize the triangulations of 2-bridge link complements.    

Futer's description of these triangulations is based on the fact that, with the exception of the unknot and the trivial link with two components, 2-bridge links can be constructed from certain 4-braids whose ends are connected.  More specifically, these 4-braids run between nested pillowcases; are formed by a sequence of $R$ and $L$ moves on the strands; and are closed off by adding crossing strands inside the inner pillowcase and outside the outer one.  Please see~\cite[Appendix]{GwithF2Bridge} for the details and a very clear exposition.

As Futer mentions, his 2-bridge link diagrams can be isotoped so the pillowcases are horizontal, that is, they are perpendicular to the plane of the page and contain the point at infinity.  The result is a diagram in which the braid strands run vertically between the bounding pillowcases, constrained to the plane of the page, except near the crossings. See Figure~\ref{fig:RandL} for a comparison of the defining $R$ and $L$ moves that occur between two pillowcases.
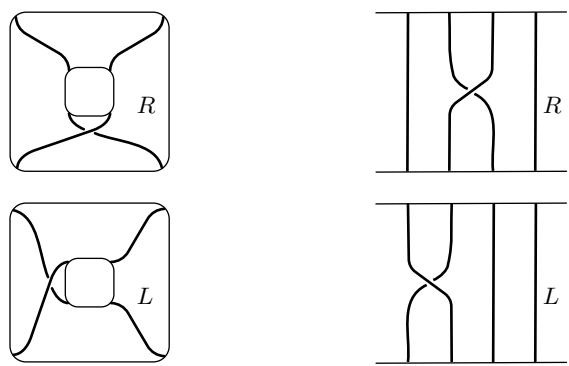
\begin{figure}[ht]\vskip.25in
\begin{overpic}[unit=.25in]{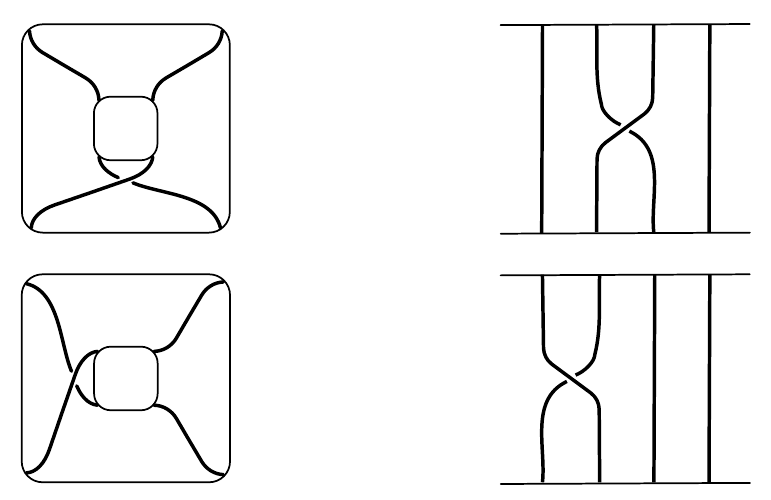}
\put(3,1.5){\small$L$} \put(11.5,1.5){\small$L$}
\put(3,5.5){\small$R$} \put(11.5,5.5){\small$R$}
\end{overpic}
\caption{The actions of $R$ and of $L$ on strands between pillowcases.  Left side is Figure~14 in~\cite[Appendix]{GwithF2Bridge}.  Right side is version with vertical strands.\label{fig:RandL}}
\end{figure}

Just as in Futer's Appendix, the 4-braid lives in a product region, $S^2\times I$, and its complement in the region is also a product, $S\times I$, where $S$ is a 4-punctured sphere.  As an example (left side of Figure~18 in~\cite[Appendix]{GwithF2Bridge}), Futer uses the 2-bridge link $K(\Omega)$ where $\Omega=R^3L^2R$, so we will too. The left side of Figure~\ref{fig:2-bridgelink} shows how to visualize $K(\Omega)$ --- the thickened curves --- in a vertical product region.  (The thinner curves will be explained later as will the right side.) 

Following Futer's techniques, we triangulate the product region using a sequence of layers of ideal tetrahedra where the top of one layer is identified to the bottom of the next.  Adopting his notation, the layers in Figure~\ref{fig:2-bridgelink} are labelled $\Delta_i$ and they each contain ideal tetrahedra $T_i$ and $T_i^\prime$.  The layer $\Delta_i$ is bounded by 4-punctured spheres, below by $S_{i}$ and above by $S_{i+1}$, with $S_{i+1}$ in $\Delta_i$  identified to $S_{i+1}$ in $\Delta_{i+1}$ by passing through a half-twist.  Numbering the punctures 0--3 will help in describing the resulting face pairings.
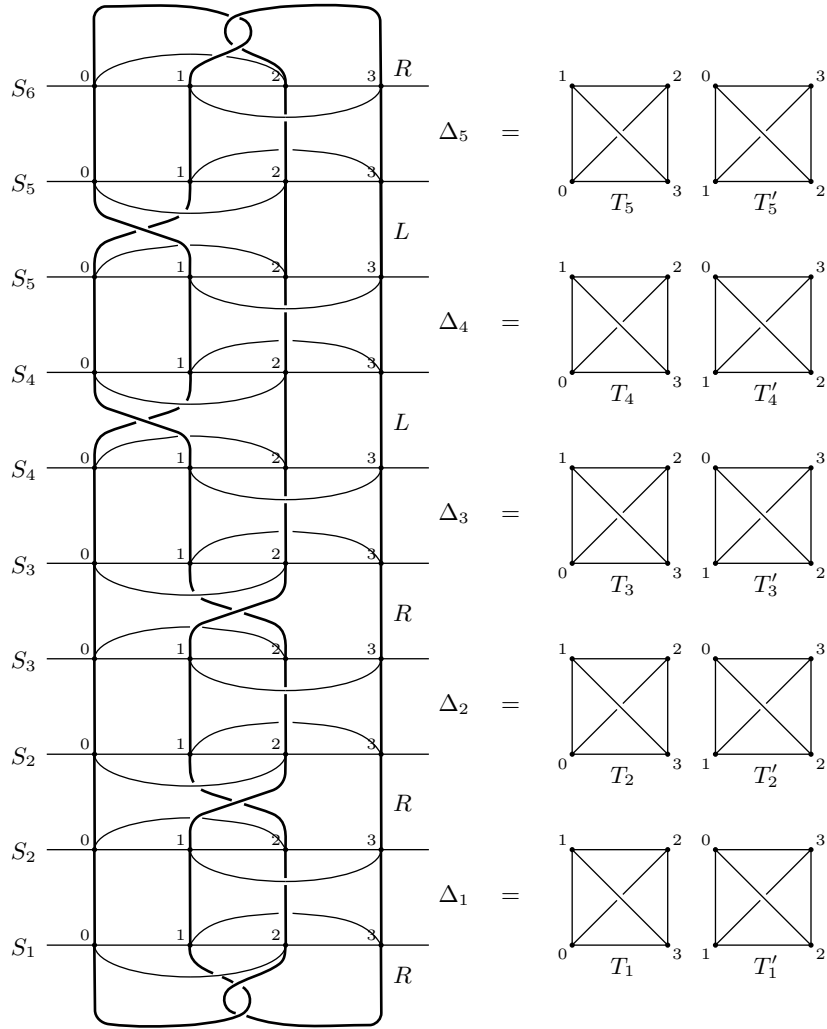
\begin{figure}[ht]\vskip.25in
\begin{overpic}[unit=.25in]{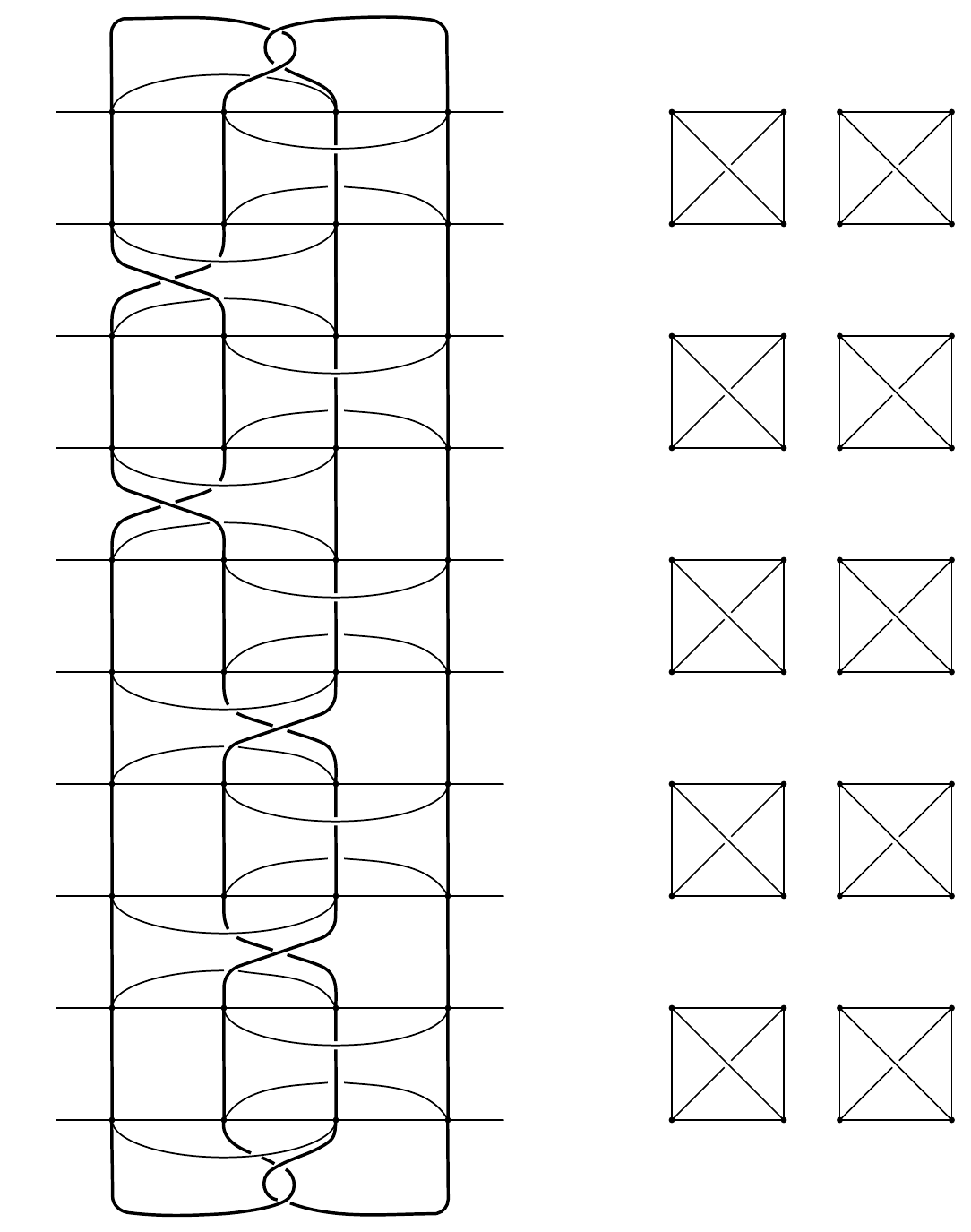}
\put(.25,19.8){\small$S_6$}\put(.25,17.8){\small$S_5$}\put(.25,15.8){\small$S_5$}\put(.25,13.8){\small$S_4$}\put(.25,11.8){\small$S_4$}\put(.25,9.8){\small$S_3$}\put(.25,7.8){\small$S_3$}\put(.25,5.8){\small$S_2$}\put(.25,3.8){\small$S_2$}\put(.25,1.8){\small$S_1$}

\put(8.25,20.2){\small$R$}\put(8.25,16.8){\small$L$}\put(8.25,12.8){\small$L$}\put(8.25,8.8){\small$R$}\put(8.25,4.8){\small$R$}\put(8.25,1.2){\small$R$}

\put(1.7, 20.1){\tiny0}\put(3.7, 20.1){\tiny1}\put(5.7, 20.1){\tiny2}\put(7.7, 20.1){\tiny3}
\put(1.7, 18.1){\tiny0}\put(3.7, 18.1){\tiny1}\put(5.7, 18.1){\tiny2}\put(7.7, 18.1){\tiny3}
\put(1.7, 16.1){\tiny0}\put(3.7, 16.1){\tiny1}\put(5.7, 16.1){\tiny2}\put(7.7, 16.1){\tiny3}
\put(1.7, 14.1){\tiny0}\put(3.7, 14.1){\tiny1}\put(5.7, 14.1){\tiny2}\put(7.7, 14.1){\tiny3}
\put(1.7, 12.1){\tiny0}\put(3.7, 12.1){\tiny1}\put(5.7, 12.1){\tiny2}\put(7.7, 12.1){\tiny3}
\put(1.7, 10.1){\tiny0}\put(3.7, 10.1){\tiny1}\put(5.7, 10.1){\tiny2}\put(7.7, 10.1){\tiny3}
\put(1.7, 8.1){\tiny0}\put(3.7, 8.1){\tiny1}\put(5.7, 8.1){\tiny2}\put(7.7, 8.1){\tiny3}
\put(1.7, 6.1){\tiny0}\put(3.7, 6.1){\tiny1}\put(5.7, 6.1){\tiny2}\put(7.7, 6.1){\tiny3}
\put(1.7, 4.1){\tiny0}\put(3.7, 4.1){\tiny1}\put(5.7, 4.1){\tiny2}\put(7.7, 4.1){\tiny3}
\put(1.7, 2.1){\tiny0}\put(3.7, 2.1){\tiny1}\put(5.7, 2.1){\tiny2}\put(7.7, 2.1){\tiny3}

\put(11.7, 20.1){\tiny1}\put(14.1, 20.1){\tiny2}\put(14.7, 20.1){\tiny0}\put(17.1, 20.1){\tiny3}
\put(11.7, 17.7){\tiny0}\put(14.1, 17.7){\tiny3}\put(14.7, 17.7){\tiny1}\put(17.1, 17.7){\tiny2}

\put(11.7, 16.1){\tiny1}\put(14.1, 16.1){\tiny2}\put(14.7, 16.1){\tiny0}\put(17.1, 16.1){\tiny3}
\put(11.7, 13.7){\tiny0}\put(14.1, 13.7){\tiny3}\put(14.7, 13.7){\tiny1}\put(17.1, 13.7){\tiny2}

\put(11.7, 12.1){\tiny1}\put(14.1, 12.1){\tiny2}\put(14.7, 12.1){\tiny0}\put(17.1, 12.1){\tiny3}
\put(11.7, 9.7){\tiny0}\put(14.1, 9.7){\tiny3}\put(14.7, 9.7){\tiny1}\put(17.1, 9.7){\tiny2}

\put(11.7, 8.1){\tiny1}\put(14.1, 8.1){\tiny2}\put(14.7, 8.1){\tiny0}\put(17.1, 8.1){\tiny3}
\put(11.7, 5.7){\tiny0}\put(14.1, 5.7){\tiny3}\put(14.7, 5.7){\tiny1}\put(17.1, 5.7){\tiny2}

\put(11.7, 4.1){\tiny1}\put(14.1, 4.1){\tiny2}\put(14.7, 4.1){\tiny0}\put(17.1, 4.1){\tiny3}
\put(11.7, 1.7){\tiny0}\put(14.1, 1.7){\tiny3}\put(14.7, 1.7){\tiny1}\put(17.1, 1.7){\tiny2}

\put(9.2,18.9){\small$\Delta_5\quad=\quad$}\put(9.2,14.9){\small$\Delta_4\quad=\quad$}\put(9.2,10.9){\small$\Delta_3\quad=\quad$}\put(9.2,6.9){\small$\Delta_2\quad=\quad$}
\put(9.2,2.9){\small$\Delta_1\quad=\quad$}

\put(12.8,17.4){\small$T_5$}\put(15.8,17.4){\small$T_5^\prime$}
\put(12.8,13.4){\small$T_4$}\put(15.8,13.4){\small$T_4^\prime$}
\put(12.8,9.4){\small$T_3$}\put(15.8,9.4){\small$T_3^\prime$}
\put(12.8,5.4){\small$T_2$}\put(15.8,5.4){\small$T_2^\prime$}
\put(12.8,1.4){\small$T_1$}\put(15.8,1.4){\small$T_1^\prime$}

\end{overpic}
\caption{The link $K(\Omega)$ where $\Omega=R^3L^2R$.  Compare left side to left side of Figure~18 in~\cite[Appendix]{GwithF2Bridge}.  Also similar to the top of Figure II.3.3 in~\cite{SakumaWeeks}.\label{fig:2-bridgelink}}
\end{figure}

The right side of Figure~\ref{fig:2-bridgelink} shows the tetrahedra forming the complement of~$K(\Omega)$.  These tetrahedra also appear in the left side of the figure.  The thin curves are their edges.  To see how this works,  examine the bottom layer, redrawn in Figure~\ref{fig:firstlayer}.  The first diagram of Figure~\ref{fig:firstlayer} (in the top left) shows an expanded version of the tetrahedra that compose~$\Delta_1$, one in which there are two copies of each edge in the plane of the page.  This expansion makes it easier to see the triangles in the bounding punctured spheres, $S_1$ and $S_2$, and how these triangles --- faces of the tetrahedra $T_1$ and $T_1^\prime$ --- live in the link complement.  The coloring indicates how the triangles in the lower $S_2$ pillowcase are identified to those in the upper one after passing through the $R$ half-twist (our version of Figure~15~\cite[Appendix]{GwithF2Bridge}):
\begin{itemize}
\item $\triangle 013$ in the front (blue) is half twisted to $\triangle 023$ in the front;
\item $\triangle 123$ in the front (green) is half twisted to $\triangle 213$ in the back;
\item $\triangle 023$ in the back (orange) is half twisted to $\triangle 013$ in the back; and
\item $\triangle 012$ in the back (pink) is half twisted to $\triangle 021$ in the front.
\end{itemize}
The face pairings defined by passing through the $L$ half-twist can be determined in the same manner, using, for example, the pillowcase~$S_4$ in Figure~\ref{fig:2-bridgelink}.

The second diagram in Figure~\ref{fig:firstlayer} shows a collapsed version of~$\Delta_1$, one consisting of two tetrahedra identified along the edges 01, 12, 23, and 30, which contains the point at infinity.  $T_1$ is in front of the page with faces 013 and 123 on top.  $T_1^\prime$ is behind the page with faces 012 and 023 on top.  The next two  diagrams also show $T_1$ and $T_1^\prime$, using arrows to indicate edge identifications.  The last one is our version of the two tetrahedra in Figure~16~\cite[Appendix]{GwithF2Bridge}, which can be separated to form the pair of tetrahedra on the bottom right in Figure~\ref{fig:2-bridgelink}.  

In each layer, the expanded version of the pair of tetrahedra on the left of Figure~\ref{fig:2-bridgelink} can be similarly associated with the flattened versions  on the right.  Thus, the tetrahedra that triangulate the product region appear in our visualization and the half-twists show how to obtain the face pairings between them.
  
\begin{figure}[ht]\vskip.25in
\begin{overpic}[unit=.25in]{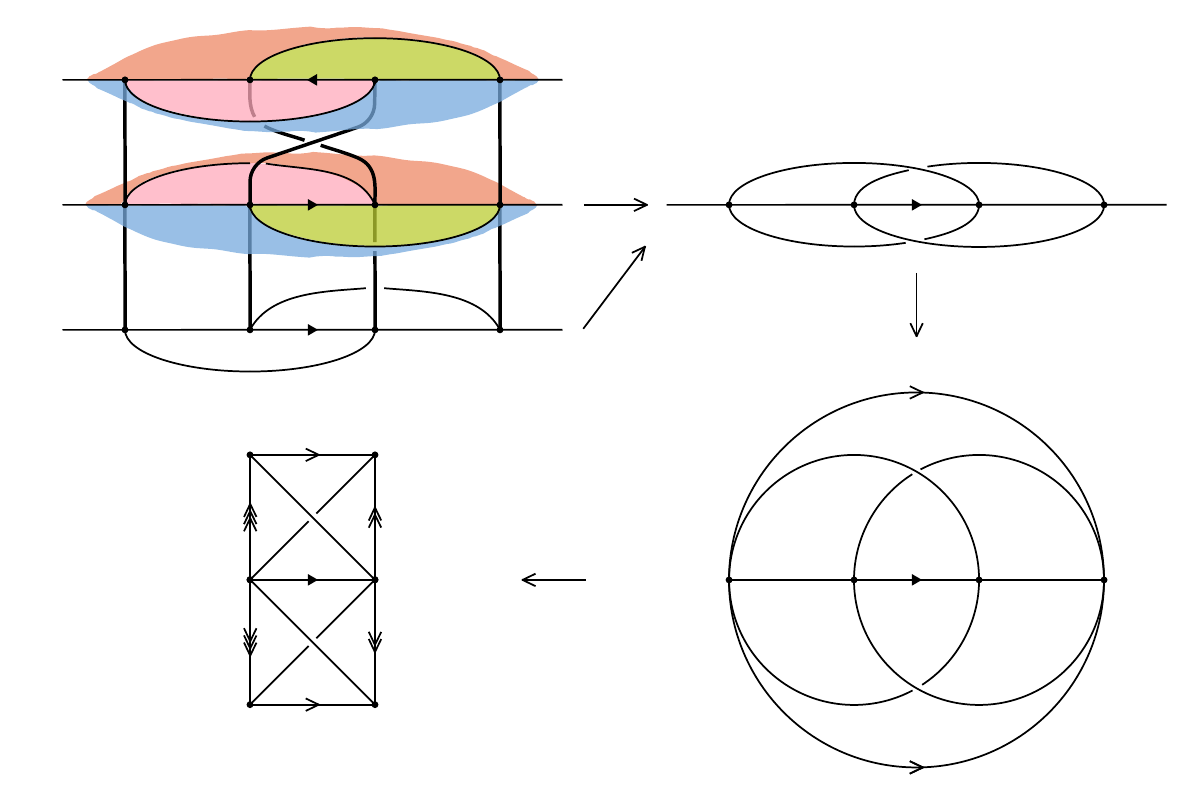}
\put(.25,11.2){\small$S_2$}\put(.25,9.2){\small$S_2$}\put(.25,7.2){\small$S_1$}

\put(8.25,8.2){\small$\Delta_1$}

\put(8.25,10.2){\small$R$}

\put(1.7, 11.4){\tiny0}\put(3.7, 11.4){\tiny1}\put(5.7, 11.4){\tiny2}\put(7.7, 11.4){\tiny3}
\put(1.7, 9.4){\tiny0}\put(3.7, 9.4){\tiny1}\put(5.7, 9.4){\tiny2}\put(7.7, 9.4){\tiny3}
\put(1.7, 7.4){\tiny0}\put(3.7, 7.4){\tiny1}\put(5.7, 7.4){\tiny2}\put(7.7, 7.4){\tiny3}

\put(11.4, 9.4){\tiny0}\put(13.4, 9.4){\tiny1}\put(15.4, 9.4){\tiny2}\put(17.4, 9.4){\tiny3}
\put(11.4, 3.4){\tiny0}\put(13.4, 3.4){\tiny1}\put(15.4, 3.4){\tiny2}\put(17.4, 3.4){\tiny3}

\put(3.7, 5.4){\tiny0}\put(6.1, 5.4){\tiny3}
\put(3.7, 3.25){\tiny1}\put(6.1, 3.25){\tiny2}
\put(3.7, 1.1){\tiny0}\put(6.1, 1.1){\tiny3}

\put(11.2,10.1){\small$T^\prime_1$}\put(11.2,8.3){\small$T_1$}
\put(11.2,4.8){\small$T^\prime_1$}\put(11.2,1.8){\small$T_1$}
\put(3.1,4){\small$T^\prime_1$}\put(3.1,2){\small$T_1$}

\end{overpic}
\caption{The layer $\Delta_1$.  Compare first figure to Figure~15 in~\cite[Appendix]{GwithF2Bridge} and last figure to Figure~16 in~\cite[Appendix]{GwithF2Bridge}.\label{fig:firstlayer}}
\end{figure}
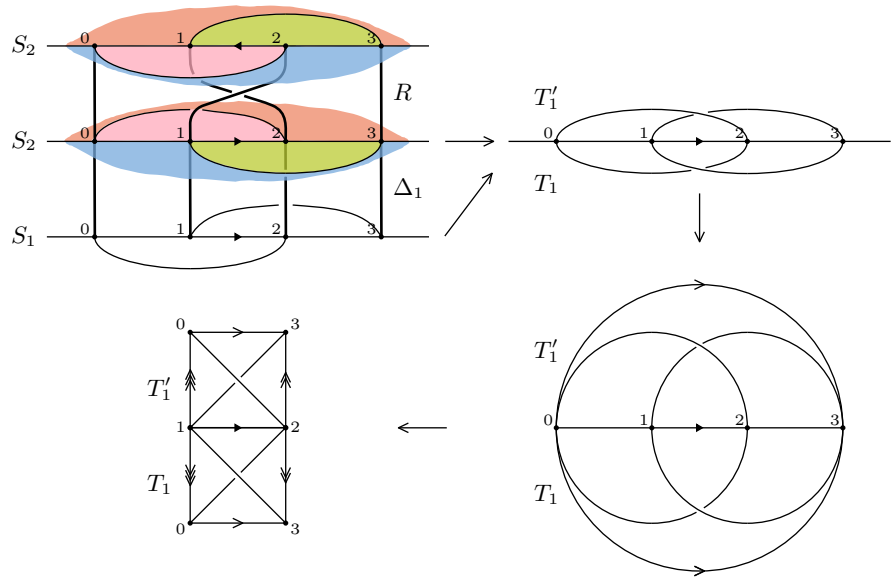

All that remains to specify a triangulation of the link complement is to show how to close off the top and bottom.  This amounts to saying how the triangles in the bounding pillowcases are identified to themselves to form clasps:  If $\Omega$ starts with~$R$, identify bottom layer triangles $\triangle 012\sim\triangle 013$ and $\triangle 023\sim\triangle 123$; If $\Omega$ starts with~$L$, identify bottom layer triangles $\triangle 012\sim\triangle 312$ and $\triangle 023\sim\triangle 013$.  The top layer triangles are identified in the same way.  For example, if  $\Omega$ ends with~$R$, $\triangle 012\sim\triangle 013$ and $\triangle 023\sim\triangle 123$.

It should not be surprising that these identifications yield diagrams  looking very much like Figure~17 in~\cite[Appendix]{GwithF2Bridge} and Figure~II.2.7 in~\cite{SakumaWeeks}.  For example, the steps forming the bottom clasp of $K(\Omega)$ where $\Omega=R^3L^2R$ are shown in Figure~\ref{fig:clasp}.  While making the specified identifications, we can follow the punctures (their paths are drawn as strands) and see that they sweep out a clasp.  Folding $S_1$ down along the 12 edge forms a pillow.  Bringing the 2 and 3 strands together identifies the gray triangles, $\triangle 012$ and $\triangle 013$.  Placing this triangle in the plane of the page means the blue 23 edge forms a belt around a new pillow with half in the front, like the 2~strand, and half behind as the 3~strand is.  The top of the pillow is a cone formed by $\triangle 123$ and the bottom by $\triangle 023$, the white triangles with the cyan markings.  Pushing  the 1 and 0 strands into the pillow flattens the cones and identifies $\triangle 123$ to $\triangle 023$, thus completing the identifications.

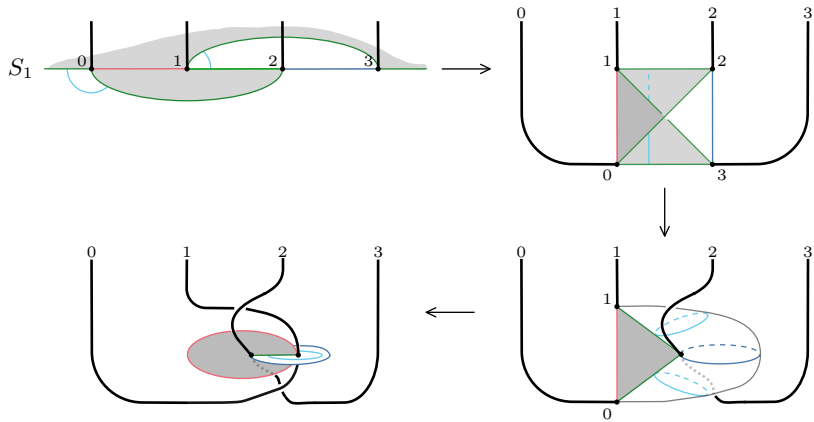
\begin{figure}[ht]\vskip.25in
\begin{overpic}[unit=.25in]{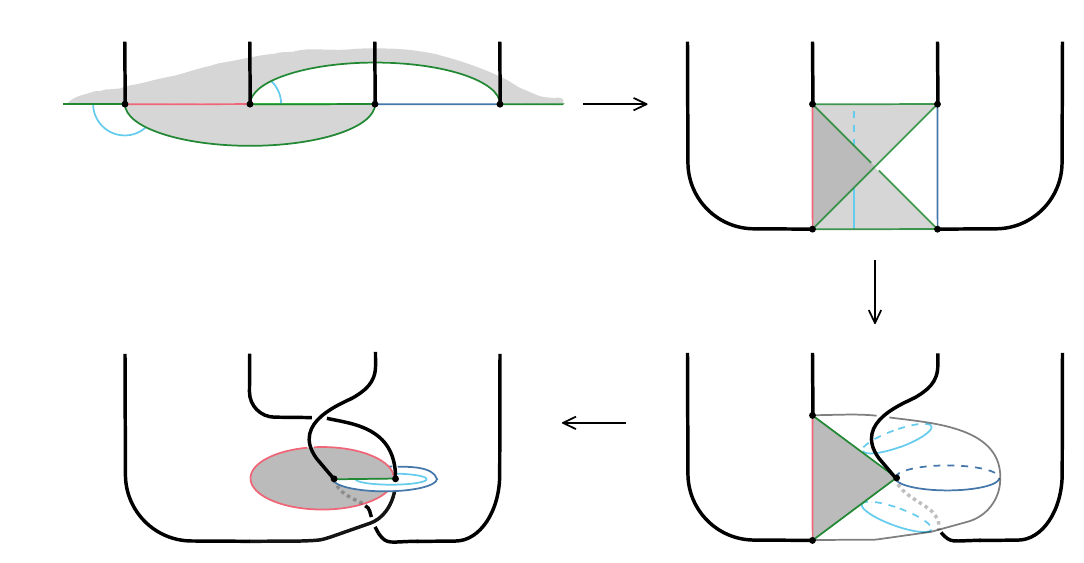}
\put(.25,7.2){\small$S_1$}

\put(1.7, 7.4){\tiny0}\put(3.7, 7.4){\tiny1}\put(5.7, 7.4){\tiny2}\put(7.7, 7.4){\tiny3}

\put(10.9, 8.4){\tiny0}\put(12.9, 8.4){\tiny1}\put(14.9, 8.4){\tiny2}\put(16.9, 8.4){\tiny3}
\put(12.7, 7.4){\tiny1}\put(15.1, 7.4){\tiny2}
\put(12.7, 5){\tiny0}\put(15.1, 5){\tiny3}

\put(10.9, 3.4){\tiny0}\put(12.9, 3.4){\tiny1}\put(14.9, 3.4){\tiny2}\put(16.9, 3.4){\tiny3}
\put(12.7, 2.4){\tiny1}
\put(12.7, 0){\tiny0}

\put(1.9, 3.4){\tiny0}\put(3.9, 3.4){\tiny1}\put(5.9, 3.4){\tiny2}\put(7.9, 3.4){\tiny3}

\end{overpic}
\caption{Forming a clasp.  Compare to Figure~17~\cite[Appendix]{GwithF2Bridge} and Figure~II.2.7 in~\cite{SakumaWeeks}.\label{fig:clasp}}
\end{figure}

Because the strands form a clasp, identifying $S_1$ to itself in such a way is equivalent to attaching a 3-ball with the needed clasp removed.  Doing the same at the top results in a triangulation of the 2-bridge link complement.  This triangulation is the same as Sakuma and Weeks' as described by Futer and is one in which the faces of the tetrahedra are easy to see in the braid diagram, as are the face pairings induced by the half-twists.

\bibliographystyle{amsplain}
\bibliography{BENBiblio} 

\end{document}